\documentclass[a4paper,11pt]{amsart}
\usepackage{amsmath,amsthm,amssymb,amscd,amsfonts}
\usepackage[matrix,arrow,curve]{xy}
\usepackage{graphicx}
\input{xypic}
\DeclareGraphicsExtensions{.jpg}

\usepackage{color}\definecolor{darkblue}{rgb}{0,0.1,.5}
\usepackage[colorlinks=true,linkcolor=darkblue, urlcolor=darkblue, citecolor=darkblue]{hyperref}
\usepackage{cleveref}

\newcommand{\hr}[2][]{\hyperref[#2]{#1~\ref{#2}}}

\newtheorem{theorem}{Theorem}[section]
\newtheorem{proposition}[theorem]{Proposition}
\newtheorem{corollary}[theorem]{Corollary}
\newtheorem{lemma}[theorem]{Lemma}

\newtheorem*{theorem*}{Theorem}
\theoremstyle{definition}
\newtheorem{definition}[theorem]{Definition}
\newtheorem{example}[theorem]{Example}
\newtheorem{construction}[theorem]{Construction}
\theoremstyle{remark}
\newtheorem*{remark}{Remark}
\newtheorem*{acknowledgements}{Acknowledgments}

\numberwithin{equation}{section}

\newcommand{\Hom}{\mathop{\mathrm{Hom}}\nolimits}
\def\id{\mathop{\mathrm{id}}}
\def\Im{\mathop{\rm Im}}
\newcommand{\Ker}{\mathop{\rm Ker}}

\def\C{\mathbb C}
\def\Q{\mathbb Q}

\def\Z{\mathbb Z}

\def\MU{\mathit{MU}}
\def\MSU{\mathit{MSU}}
\def\BU{\mathit{BU}}
\def\BSU{\mathit{BSU}}
\def\SU{\mathit{SU}}
\def\cf{c^{\scriptscriptstyle U}}

\def\du{d^{\scriptscriptstyle U}}
\def\bdu{\overline{d}\phantom{d}\!\!\!^{\scriptscriptstyle U}}

\def\ge{\geqslant}
\def\geq{\geqslant}
\def\le{\leqslant}

\def\pt{\mathit{pt}}

\begin{document}

\title[$SU$-linear operations in complex cobordism]{$SU$-linear operations in complex cobordism and the $c_1$-spherical bordism theory}

\author{Georgy Chernykh}
\address{Faculty of Mathematics and Mechanics, Moscow
State University, Russia;\newline\indent
Steklov Mathematical Institute of the Russian Academy of Sciences, Moscow, Russia}
\email{aaa057721@gmail.com}

\author{Taras Panov}
\address{Faculty of Mathematics and Mechanics, Moscow State University, Russia;
\newline\indent
HSE University, Moscow, Russia;\newline\indent
Institute for Information Transmission Problems, Russian Academy of Sciences, Moscow}
\email{tpanov@mech.math.msu.su}
\urladdr{http://higeom.math.msu.su/people/taras/}

\thanks{The study was funded within the framework of the HSE University Basic Research Program and by the Russian Foundation for Basic Research (grant no.~20-01-00675). G.\,Chernykh was also supported by Theoretical Physics and Mathematics Advancement Foundation ``BASIS''}


\begin{abstract}
We study the $SU$-linear operations in complex cobordism and prove that they are generated by the well-known geometric operations~$\partial_i$. For the theory $W$ of $c_1$-spherical bordism, we describe all $SU$-linear multiplications on $W$ and projections $\MU \to W$. We also analyse complex orientations on $W$ and the corresponding formal group laws $F_W$. The relationship between the formal group laws~$F_W$ and the coefficient ring $W_*$ of the $W$-theory was studied by Buchstaber in 1972. We extend his results by showing that for any $SU$-linear multiplication and orientation on $W$,  the coefficients of the corresponding formal group law $F_W$ do not generate the ring~$W_*$, unlike the situation with complex bordism.
\end{abstract}

\maketitle


\section*{Introduction}
\emph{Complex bordism}, or \emph{$U$-bordism}, is the bordism theory of stably complex manifolds. Geometrically, a stably complex structure ($U$-structure) on a manifold $M$ is a choice of a complex structure on its stable tangent bundle, or a reduction of the structure group of the stable tangent bundle to the unitary group~$U=U(\infty)$. Homotopically, a stably complex structure is the homotopy class of a lift of the map $M\to BO$ classifying the stable tangent bundle to a map $M\to BU$. The bordism classes of stably complex manifolds form  a graded ring with respect to the operations of disjoint union and cartesian product, called the \emph{complex bordism ring} and denoted by~$MU_*$. It is the coefficient ring of the \emph{complex bordism theory}, the generalised (co)homology theory defined by the \emph{Thom spectrum} $\MU=\{\MU(n)\}$, where $MU(n)$ is the Thom space of the universal $U(n)$-bundle $EU(n)\to BU(n)$. Given a CW-pair $(X,A)$, its bordism and cobordims groups are defined by
\[
\begin{aligned}
  MU_n(X,A)&=\lim_{k\to\infty}\pi_{2k+n}\bigl((X/A)\wedge\MU(k)\bigr),\\
  MU^n(X,A)&=\lim_{k\to\infty}\bigl[\varSigma^{2k-n}(X/A),\MU(k)\bigr]
   \quad \text{for a finite CW-pair $(X,A)$}.
\end{aligned}
\]
In particular, $  MU_*=\pi_*(\MU)=MU_*(\pt)=
\lim_{k\to\infty}\pi_{2k+*}\bigl(\MU(k)\bigr)$. We also denote $MU^*=MU^*(\pt)$, the \emph{complex cobordism ring}, graded nonpositively.

\emph{$SU$-bordism} is the bordism theory of smooth manifolds with a special unitary structure in the stable tangent bundle. Geometrically, an $SU$-structure on a manifold $M$ is defined by a reduction of the structure group of the stable tangent bundle of~$M$ to the group $SU(N)$. Homotopically, an $SU$-structure is the homotopy class of a lift of the classifying map 
$M\to BO(2N)$ to a map $M\to BSU(N)$. A manifold $M$ admits an $SU$-structure whenever it admits a stably complex structure with $c_1(\mathcal TM)=0$. The \emph{$SU$-bordims ring} $\MSU_*=\pi_*(\MSU)$ is the coefficient ring of the \emph{$\SU$-bordism theory}, defined by the Thom spectrum 
$\MSU=\{\MSU(n)\}$.

The details of the construction of the $\SU$-bordism and the description of the coefficient ring $\MSU_*$ can be found in~\cite{novi67,ston68,c-l-p19}.

A (stable) \emph{operation} $f$ of degree $n$ in complex cobordism is a family of additive maps
\[
  f\colon\MU^k(X,A)\to\MU^{k+n}(X,A)
\] 
which are functorial with respect to $(X,A)$ and commute with the suspension isomorphisms. The set of all operations is an $\MU^*$-algebra, denoted by~$A^U$. It can be identified with the self-maps of the $\MU$ spectrum:
\[
  A^U\cong[\MU,\MU]_*=MU^*(\MU)=\lim\limits_{\longleftarrow}\MU^{*+2N}(\MU(N)).
\]
There is an isomorphism of left $\MU^*$-modules\\[-2mm] 
\[
  A^U\cong \MU^*\mathbin{\widehat\otimes} S, 
\]
where $S$ is the \emph{Landweber--Novikov algebra}, generated by the operations $S_\omega=\varphi^*(s^{\scriptscriptstyle U}_\omega)$ corresponding via the Thom isomorphism $\varphi^*$ to the universal characteristic classes $s^{\scriptscriptstyle U}_\omega\in\MU^*(\BU)$ defined by symmetrising the monomials $t_1^{i_1}\cdots t_k^{i_k}$ indexed by partitions $\omega=(i_1,\ldots,i_k)$. Therefore, any element $a\in A^U$ can be written uniquely as an infinite series $a = \sum_{\omega} \lambda_{\omega} S_{\omega}$ where $\lambda_{\omega}\in\MU^*$. The Hopf algebra structure of $S$ is described in~\cite{land67} and~\cite[\S5]{novi67}. 

The spectrum $\MU$ is an $\MSU$-module via the forgetful morphism $\MSU\to\MU$, and an operation $f\colon\MU\to\MU$ is \emph{$\MSU$-linear} if it is an $\MSU$-module map. By the standard property of spectra with torsion-free homotopy and homology groups, the $\MSU$-linearity of an operation $f\colon\MU\to\MU$ can be detected by its action on the coefficients $MU_*=\pi_*(\MU)$. Namely, an operation $f$ is $\SU$-linear if and only if it satisfies $f(ab)=a f(b)$ for any $a\in\MSU_*$, $b\in\MU_*$ (see Theorem~\ref{su-linmu}).

A family of geometric operations $\partial_i\in[\MU,\MU]_{-2i}=[\MU,\Sigma^{2i}\MU]$ was defined by Conner and Floyd~\cite{co-fl66} and studied further by Novikov~\cite{novi67}. The operation $\partial_i$ assigns to a complex bordism class $[M]\in \MU_{2n}$ the bordism class of the submanifold $M_i\subset M$ dual to~$(\det\mathcal T M)^{\oplus i}$ (the $i$-fold direct sum of the determinant of the tangent bundle). In particular, $\partial_1=\partial\colon \MU_{2n}\to \MU_{2n-2}$ is the ``boundary operation'' that takes $[M]$ to the bordism class of the submanifold dual to $c_1(\mathcal TM)$. Clearly, $\partial[M]$ belongs to the image of the forgetful map $\MSU_* \to \MU_*$. Furthermore, the operations $\partial_i$ are $SU$-linear by inspection. 

In Section~\ref{linsect} we describe the algebra of $\SU$-linear operations in complex cobordism. We show that the operations $\partial_i$, $i=1,2,\ldots$, form a topological basis of the left $MU^*$-module of $\SU$-linear operations. That is, any $SU$-linear operation $f \in [\MU, \MU]_{\MSU,*}$ can be written uniquely as a series $f=\sum_{i \ge 0} \mu_i \partial_i$ with $\mu_i \in \MU^{-2i-*}$, see Theorem~\ref{su-lin}. In Theorem~\ref{comp} we describe the product structure of $\SU$-linear operations under composition via the coefficients of the formal group law in complex cobordism.

Conner and Floyd~\cite{co-fl66} and Stong~\cite{ston68} defined \emph{$c_1$-spherical bordism} $W$, an intermediate theory between the $\SU$- and $U$-bordism, following a similar construction of Wall in oriented bordism. The theory $W$ was a key technical tool in Conner and Floyd's calculation of torsion in $SU$-bordism. Both~\cite{co-fl66} and~\cite{ston68} defined a multiplicative structure on $W$ using an $\SU$-linear projection $\pi\colon\MU\to W$. Stong~\cite{ston68} showed that the coefficient ring of the theory $W$ is polynomial with respect to the multiplication defined by his projection. Although Conner--Floyd and Stong defined their projections differently, in the subsequent literature on $\SU$- and $c_1$-spherical bordism the two projections were used interchangeably, suggesting that the two may coincide. As shown in~\cite[\S6]{c-l-p19}, the Conner--Floyd and Stong projections are different, despite defining the same multiplication on~$W$ (see Example~\ref{CFproj}). 

In Section~\ref{Wsect} we give several descriptions of projections $\pi\colon\MU\to W$ and give the conditions specifying $SU$-linear projections. We express the $SU$-linear Stong projection $\pi_0\colon\MU\to W$ as a series in operations $\partial_i$ via the coefficients of the formal group law in complex cobordism (Proposition~\ref{pi}), and show that any other projection $\pi\colon\MU\to W$ has the form $\pi_0(1+f\varDelta)$ for some operation $f\in[\MU,\Sigma^{-4}\MU]$ (Theorem~\ref{pr1}), where $\varDelta\in[\MU,\Sigma^{4}\MU]$ is the Conner--Floyd operation satisfying $W=\Ker\varDelta$. In this description, $\SU$-linear projections correspond to $\SU$-linear operations~$f$. In Theorem~\ref{multgen} we describe all $\SU$-linear multiplications on $c_1$-spherical bordism $W$, and specify a condition for the multiplication to be defined by an $\SU$-linear projection.

In Section~\ref{fgsect} we study complex orientations of the theory $W$ and the corresponding formal group laws. Any complex orientation $w\in\widetilde W^2(\C P^{\infty})$ is obtained by applying an $SU$-linear projection to an orientation 
$\widetilde u \in \widetilde {\MU} \vphantom{MU}^2(\C P^{\infty})$ of complex cobordism (Proposition~\ref{worie}). A multiplication on $W$ together with a complex orientation $w$ define a formal group law $F_W\in W_*[[u,v]]$. This formal group law was studied by Buchstaber in~\cite{buch72}, where it was stated that the coefficients of $F_W$ do not generate the ring $W_*$, unlike the situation with complex cobordism. We give a complete proof of this claim in Theorem~\ref{notgen}. It is shown that the coefficients of $F_W$ do not generate
$W_*$ for any multiplication and complex orientation of $W$, not just for the standard ones defined by the Stong projection. We also prove another statement from~\cite{buch72}: after inverting $2$, the ring generated by the coefficients of $F_W$ for some orientation $w$ coincides with $W_*[\frac12]$ (Theorem~\ref{1p}).

It would be interesting to give a geometric construction for the multiplicative transformation (genus) $\MU\to W$ classifying the formal group law $F_W$, for example, in terms of the geometric generators of $W_*$ and $\MSU_*[\frac12]$ described in~\cite[Part~II]{c-l-p19}. A related construction of polynomial generators of $\MSU_*[\frac12]$ associated with the classifying maps for the Abel, Buchstaber and Krichever formal group laws was recently given by Bakuradze~\cite{baku} (the Krichever formal group law gives rise to the Krichever--Hoehn complex elliptic genus).

\begin{acknowledgements}
The authors benefitted greatly from the advice and encouragement of Victor Buchstaber. We thank Tom Bachmann for the inspiring discussion of $\SU$-linear operations in complex cobordism, and in particular for his question on whether the geometric operations of Conner and Floyd form a topological basis of $\SU$-linear operations, which resulted in our Theorem~\ref{su-lin}.  We also thank the referee for valuable comments and corrections.
\end{acknowledgements}

\section{$\SU$-linear operations in complex cobordism}\label{linsect}

We are interested in cohomology operations and their linearity properties with respect to module pairings in cohomology theories. Therefore, we work in the stable homotopy category (see~\cite{adam74, swit75, marg83, rudy98, ba-ro20}) and do not use any strict models for spectra. All modules, rings and linear properties (as in Definition~\ref{lin}) below are considered in the homotopy sense. We only use the monoidal structure of the stable homotopy category and \emph{do not} use any strict monoidal model categories of spectra (see however the remark after Proposition~\ref{bij}). Whenever we refer to a spectrum or a map (between spectra) it means an object or a morphism in the stable homotopy category.

\subsection{Equivalent definitions of $\SU$-linearity}\label{rlinsect}

Given spectra $X$ and $Y$, we denote by $[X, Y]$ the set of morphisms (in the stable homotopy category) between them, which is an abelian group, and denote by $\pi_*(X)$ the homotopy groups of~$X$.

\begin{definition}\label{lin}
Let $R$ be a commutative ring spectrum, and let $E$ and $F$ be $R$-modules. Consider the following properties of a morphism $f \in [E, F]$:
\begin{itemize}
\item[(a)] $f$ is \emph{$R$-linear}, that is, the following square is commutative (in the stable homotopy category)
$$
\xymatrix{
R \wedge E \ar[r]^{1 \wedge f} \ar[d] & R \wedge F \ar[d] \\
E \ar[r]^{f} & F
}
$$

\item[(b)] $f$ is {\itshape linear with respect to the multiplication by elements of} $\pi_*(R)$, that is, for any $r \in \pi_k(R)$ the following diagram of spectra is commutative (in the stable homotopy category)
$$
\xymatrix{
\Sigma^k E \ar[r]^{r\cdot} \ar[d]_{\Sigma^k f} & E \ar[d]^{f} \\
\Sigma^k F \ar[r]^{r\cdot} & F
}
$$
\item[(c)] $\pi_*(f) \in \Hom (\pi_*(E), \pi_*(F))$ is $\pi_*(R)${\itshape-linear}.
\end{itemize}
\end{definition}

\begin{proposition}
Property (a) implies (b), and (b) implies (c).
\end{proposition}
\begin{proof}(a)$\Rightarrow$(b): The diagram of~(b) is expanded as follows:
$$
  \xymatrix{
  S^k \wedge E \ar[r]^{r \wedge 1} \ar[d]_{1 \wedge f} & R \wedge E \ar[r] \ar[d] ^{1 \wedge f} & 
  E   \ar[d]^{f} \\ S^k \wedge F \ar[r]^{r \wedge 1} & R \wedge F \ar[r] & F 
  }
$$
The left square commutes, and the commutativity of the right square is asserted by~(a). So the whole diagram also commutes, which implies~(b).

\smallskip

(b)$\Rightarrow$(c): For $r \in \pi_k(R)$ and $a \in \pi_n(E)$, the condition $\pi_*(f)(ra)=r\pi_*(f)(a)$ is expressed as the commutativity of the outer diagram:
$$
\xymatrix{
  &R \wedge E \ar[r] & E \ar[dr]^f & \\
  S^k \wedge S^n \ar[ur]^{r \wedge a} \ar[dr]^{1 \wedge a} & & & F \\
  & S^k \wedge E \ar[r]^{r \wedge f} \ar[uu]_{r \wedge 1} & R \wedge F \ar[ur]
  }
$$
Here the left triangle is commutative, and the commutativity of the right part is asserted by~(b).
\end{proof}

Alongside with $[E,F]$ we consider the graded abelian group $[E, F]_*$ with graded components given by $[E, F]_k = [\Sigma^kE,F]=[E,\Sigma^{-k}F]=F^{-k}(E)$, $k\in\Z$. 

We are interested in $\SU$-linear operations $f\in[\MU,\MU]_*$ in complex cobordism, which correspond to $R=\MSU$, $E=\MU$ and $F=\Sigma^k \MU$  in the notation of Definition~\ref{lin}. In this case the three versions of $SU$-linearity agree:

\begin{theorem}\label{su-linmu}
The three conditions in Definition~\ref{lin} are equivalent for $\SU$-linear operations $f \in [\MU, \MU]_*$ in complex cobordism.
\end{theorem}

The proof consists of three lemmata. First, there is the following useful result relating a morphism of spectra to its action on the homotopy groups.

\begin{lemma}[{\cite[Lemma~VII.3.2]{rudy98}}]\label{monic}\ 
Let $E$ and $F$ be connective spectra of finite type such that $H_*(E)$ and 
$\pi_*(F)$ are torsion-free. Then the natural map 
\[
  p \colon [E, F] \to \Hom(\pi_*(E), \pi_*(F))
\]   
is injective.
\end{lemma}

Lemma~\ref{monic} can be generalised to morphisms involving three spectra.
Given three spectra $E,F,G$, there is a natural map $q \colon [E \wedge F, G] \to \Hom (\pi_* (E) \otimes \pi_*(F), \pi_*(G))$ defined as follows: for $f \in [E \wedge F, G]$ and $a \in \pi_k(E), \, b \in \pi_n(F)$, the element $q(f)(a,b) \in \pi_{k+n}(G)$ is represented by the map $S^k \wedge S^n \xrightarrow{a \wedge b} E \wedge F \xrightarrow{f} G$.

\begin{lemma} \label{pmonic}
Let $E$, $F$ and $G$ be connective spectra of finite type such that $H_*(E)$, $H_*(F)$ and 
$\pi_*(G)$ are torsion-free. 
Then the natural map 
\[
  q \colon [E \wedge F, G] \to \Hom (\pi_* (E) \otimes \pi_*(F), \pi_*(G))
\]  
is injective.
\end{lemma}

\begin{proof}
Applying Lemma~\ref{monic} to the spectra $E\wedge F$ and $G$ we obtain that the map
\[
  [E \wedge F, G] \to \Hom (\pi_* (E\wedge F), \pi_*(G))
\]  
is injective. We therefore need to prove the injectivity of the map
\[
  \Hom (\pi_* (E\wedge F), \pi_*(G))\to \Hom (\pi_* (E)\otimes\pi_*(F), \pi_*(G))
\]
induced by the natural map $\pi_* (E)\otimes\pi_*(F)\to\pi_*(E\wedge F)$.

We have the following commutative square
\[
\xymatrix{
  \Hom (\pi_* (E\wedge F), \pi_*(G)) \ar[r] \ar[d] & \Hom (\pi_* (E)\otimes\pi_*(F),   
  \pi_*(G)) \ar[d] \\
  \Hom (\pi_* (E\wedge F) \otimes \Q, \pi_*(G) \otimes \Q) \ar[r] & \Hom (\pi_* (E) 
  \otimes\pi_*(F) \otimes \Q, \pi_*(G) \otimes \Q)
}
\]
Here, the left map is the composite of the map $\Hom (\pi_* (E\wedge F), \pi_*(G)) \to \Hom (\pi_* (E\wedge F), \pi_*(G) \otimes \Q)$ induced by the homomorphism $\pi_*(G) \to \pi_*(G)\otimes \Q$, which is injective since $\pi_*(G)$ is torsion-free, and the natural isomorphism $\Hom (\pi_* (E\wedge F), \pi_*(G)\otimes \Q) \cong \Hom (\pi_* (E\wedge F) \otimes \Q, \pi_*(G) \otimes \Q)$. Hence, the left map in the diagram above is injective.

The bottom map is induced by the map $\pi_*(E) \otimes \pi_*(F) \otimes \Q \to \pi_*(E \wedge F) \otimes \Q$, which is an isomorphism. Indeed, it is a natural transformation of homology theories $\pi_*(-) \otimes \pi_*(F) \otimes \Q$ and $\pi_*(- \wedge F) \otimes \Q$, which is an isomorphism on the sphere spectrum. Hence, the bottom map in the diagram above is an isomorphism.

It follows that the top map in the commutative square above is injective, as needed.
\end{proof}

\begin{remark}
By induction, the statement of Lemma~\ref{pmonic} can be generalised to the smash product of an arbitrary number of spectra.
\end{remark}

\begin{lemma}\label{r-lin}
In the notation of Definition~\ref{lin},
let $R$, $E$ and $F$ be connective spectra of finite type such that $H_*(R)$, $H_*(E)$ and $\pi_*(F)$ are torsion-free. 
Then property (c) implies~(a).
\end{lemma}

\begin{proof}
Suppose $f \in [E, F]$ satisfies property (c) of Definition~\ref{lin}. Consider the morphisms $\varphi_1 \colon R \wedge E \xrightarrow{1 \wedge f} R \wedge F \to F$, $\varphi_2 \colon R \wedge E \to E \xrightarrow{f} F$ and set $\varphi=\varphi_1-\varphi_2$. We need to prove that $\varphi=0$.

Applying the map $q \colon [R \wedge E, F] \to \Hom (\pi_* (R) \otimes \pi_*(E), \pi_*(F))$ to $\varphi \colon R \wedge E \to F$ we obtain $q(\varphi)(r, a) = r f_*(a) - f_*(ra)$, which is zero for any $r \in \pi_*(R)$, $a \in \pi_*(E)$ by property~(c). Then Lemma~\ref{pmonic} implies that $\varphi = 0$.
\end{proof}

\begin{proof}[Proof of Theorem~\ref{su-linmu}]
The result follows by setting $R=\MSU$, $E=\MU$ and $F =\Sigma^k  \MU$ in Lemma~\ref{r-lin}.
\end{proof}

\subsection{Description of $\SU$-linear operations in complex cobordism}\label{sulinsect}

The following proposition is contained implicitly in Conner and Floyd \cite[Chapter III]{co-fl66}, and its real version in Atiyah~\cite[Proposition~4.1]{ati61}.

\begin{proposition}\label{msu-lin}
There is an equivalence of $\MSU$-modules
\[
  \MU\simeq \MSU \wedge \Sigma^{-2} \mathbb C P^\infty.
\]
\end{proposition}

\begin{proof}
Since $\mathbb C P^\infty=\MU(1)$, we have an $\MSU$-module map 
\[
  \MSU \wedge \Sigma^{-2} \mathbb C P^\infty=\MSU \wedge \Sigma^{-2} MU(1) \xrightarrow{1 \wedge i} \MSU \wedge \MU \to \MU.
\]
It is a map of Thom spectra induced by the map of spaces 
\[
  \BSU \times \BU(1) \to \BU \times \BU \xrightarrow{\oplus} \BU,
\]
This composite map is a homotopy equivalence, with the homotopy inverse given by 
\[
  \BU\to \BSU \times \BU(1),\quad \eta\mapsto (\eta - \det \eta) \times \det \eta.
\]
Hence, it induces an equivalence of the corresponding Thom spectra.
\end{proof}

Given $R$-module spectra $E,F$, denote by $[E, F]_R$ the abelian group of $R$-linear morphisms. For a free $R$-module $R \wedge X$, there is a natural isomorphism between $[R \wedge X, F]_R$ and $[X, F]$, which is defined as follows. A morphism $X \xrightarrow{f} F$ corresponds to an $R$-linear morphism $R \wedge X \xrightarrow{1 \wedge f} R \wedge F \to F$. Conversely, an $R$-linear morphism 
$R \wedge X \xrightarrow{g} F$ corresponds to a morphism $X \simeq S \wedge X \xrightarrow{e \wedge 1} R \wedge X \xrightarrow{g} F$, where $e \colon S \to R$ is the unit of~$R$. 

We have $\MU^*(\C P^{\infty})=\MU^*[[u]]$, where $\MU^*=\MU^*(\pt)$ is the cobordism ring of point, $u=\cf_1 \in \widetilde {\MU} \vphantom{MU}^2 (\C P^{\infty})$ is the canonical orientation (the universal first Conner--Floyd class) defined by the hyperplane section $\C P^{\infty-1}\subset\C P^\infty$. Elements of $\widetilde {\MU} \vphantom{MU}^* (\C P^\infty)$ are represented by power series $f(u)$ in $u$ with zero constant term. 

\begin{proposition}\label{bij}
The abelian group $[\MU, \MU]_{\MSU,\,k}$ of $\SU$-linear operations  is isomorphic to $\widetilde {\MU} \vphantom{MU}^{2-k} (\C P^\infty)$. More precisely, if $u \in \widetilde {\MU} \vphantom{MU}^2(\C P^{\infty})$ is the canonical orientation in complex cobordism, then the map
\[
  [\MU, \MU]_*\to\widetilde {\MU} \vphantom{MU}^{2-*} (\C P^\infty),\quad f\mapsto f(u),
\]
becomes an isomorphism when restricted to the subgroup $[\MU,\MU]_{\MSU,*}$ of $SU$-linear operations.
\end{proposition}

\begin{proof}
We have
\[
  [\MU, \MU]_{\MSU,*}\simeq[\MSU \wedge \Sigma^{-2} \MU(1), \MU]_{\MSU,*}\simeq
  [\Sigma^{-2} \MU(1), \MU]_*=\widetilde {\MU} \vphantom{MU}^{2-*} (\C P^\infty),
\]
where the first isomorphism follows from Proposition~\ref{msu-lin}, and the second by the discussion preceding Proposition~\ref{bij}. Under these isomorphisms, an $\SU$-linear operation 
$f \colon \MU \to \Sigma^{-k} \MU$ is mapped to the composite
\[
  \Sigma^{-2} \MU(1) \simeq S\wedge \Sigma^{-2} \MU(1) \xrightarrow{e \wedge 1} \MSU \wedge   
  \Sigma^{-2} \MU(1) \xrightarrow{1 \wedge i} \MSU \wedge \MU \to \MU \xrightarrow{f} 
  \Sigma^{-k} MU
\]
Since $f$ is $\SU$-linear, we can rewrite this composite by changing the last two maps as follows:
\[
  S\wedge \Sigma^{-2} \MU(1) \xrightarrow{e \wedge 1} \MSU \wedge \Sigma^{-2} \MU(1)   
  \xrightarrow{1 \wedge i} \MSU \wedge \MU \xrightarrow{1 \wedge f} \MSU \wedge \Sigma^{-k} \MU   
  \to \Sigma^{-k} \MU
\]
Since $e$ is the unit, the composite above is $\Sigma^{-2} \MU(1) \xrightarrow{i} \MU \xrightarrow{f} \Sigma^{-k} \MU$. Now the map $\Sigma^{-2} \MU(1) \xrightarrow{i} \MU$ represents the canonical orientation $u \in \widetilde {\MU} \vphantom{MU}^2(\C P^{\infty})$, so the composite represents $f(u) \in \widetilde{\MU} \vphantom{MU}^{2-k}(\C P^{\infty})$, as claimed.
\end{proof}

\begin{remark}\label{strict}
Although we work in the stable homotopy category, the bordism spectra such as $\MU$ and $\MSU$ have strictly commutative models. More precisely, there is a symmetric monoidal category of strict $\MSU$-modules $\mathcal 
M_{\MSU}$ and its homotopy (derived) category  $\mathcal D_{\MSU}$ (see~\cite{ekmm}). The same argument as in Proposition~\ref{msu-lin} shows that there is an equivalence of strict $\MSU$-modules $\MU\cong\MSU \wedge \Sigma^{-2} \C P^{\infty}$, which together with the adjunction equivalence of~\cite[Proposition~III.4.1]{ekmm} gives an equivalence
\begin{multline*}
  \mathcal D_{\MSU}(\MU, \MU) \cong \mathcal D_{\MSU}
  (\MSU \wedge   \Sigma^{-2} \C P^{\infty}, \MU) 
  \\ \cong \mathcal D_{S}(\Sigma^{-2} \C P^{\infty}, \MU) = 
  [\Sigma^{-2} \C P^{\infty}, \MU].
\end{multline*}
Hence, the result of Proposition~\ref{bij} holds for the set of homotopy classes of strict $\MSU$-module self-maps of~$\MU$. Nevertheless, since our primary motivation lies in the study of cohomology operations that are $\SU$-linear in the sense of pairings in cohomology theories, we consider self-maps of 
$\MU$ that are $\SU$-linear only up to homotopy.
\end{remark}

By Proposition \ref{bij}, a series $f(u) \in \widetilde{\MU} \vphantom{MU}^*(\C P^{\infty})$ determines uniquely an $\SU$-linear operation~$f$.
Therefore, to complete the description of $\SU$-linear operations in complex cobordism we need to describe the $\SU$-linear operations acting on $\MU^* (\C P^\infty)$ by $u\mapsto \sum_{i \ge 0} \lambda_i u^{i+1}$, where $\sum_{i \ge 0} \lambda_i u^{i+1}\in\widetilde {\MU} \vphantom{MU}^{2-2k}(\C P^\infty)$ is an arbitrary series with $\lambda_i\in\MU^{-2i-2k}$.
The required operations were introduced by Novikov in~\cite{novi67}. We recall their definition following~\cite{c-l-p19}.

The stable Thom isomorphism
\[
  \varphi^*\colon\MU^n(\BU) \xrightarrow{\cong}\MU^n(\MU)=[\MU,\MU]_{-n}
\]
identifies universal characteristic classes with cohomological operations in complex cobordism.

Given two complex line bundles $\xi$, $\eta$ with $u=\cf_1(\xi)$, $v=\cf_1(\eta)$, the first Conner--Floyd class of their tensor product is expressed as a power series in $u,v$ with coefficients in $\MU^*$, known as the \emph{formal group law of geometric cobordisms}~\cite{novi67}, \cite{buch12}, or simply the \emph{formal group law in complex cobordism}:
\begin{equation}\label{fglcc}
  F_U(u,v)=\cf_1(\xi\otimes\eta)=u+v+\sum_{i\ge1,\,j\ge 1}\alpha_{ij}u^iv^j,\quad
  \alpha_{ij}\in \MU^{-2(i+j-1)}.
\end{equation}
We denote by $\overline u$ the inverse series of $u$ with respect to $F_U$; it satisfies $F_U(u,\overline u)=0$.

\begin{construction}[operations $\varDelta_{(k_1, k_2)}$]\label{partial_i}
Consider the universal characteristic class $\du\in \MU^2(BU)$ given by $\du(\xi)=\cf_1(\det\xi)$. We also set $\bdu=\cf_1(\overline{\det\xi})$, which can be identified with the inverse series of $\du$ and satisfies $F_U(\du,\bdu)=0$.

Given non-negative integers $k_1$, $k_2$, define the operation
\[
  \varDelta_{(k_1, k_2)} = \varphi^* \bigl((\bdu)^{k_1}(\du)^{k_2}\bigr)
  \in[\MU, \MU]_{-2(k_1+k_2)}
\]

Geometrically, 
the operation 
$\varDelta_{(k_1, k_2)}$ takes a bordism class $[M]\in\MU_*$ to the class $[M_{k_1, k_2}]$ of a submanifold dual to $(\det\mathcal TM)^{\oplus k_1} \oplus (\overline {\det\mathcal TM})^{\oplus k_2}$.

The action of $\varDelta_{(k_1, k_2)}$ on the canonical orientation 
$u \in \widetilde {\MU} \vphantom{MU}^2(\C P^{\infty})$ is given by
\begin{equation}\label{Deltau}
\varDelta_{(k_1, k_2)} u = u\, \overline u ^{k_1}u^{k_2}
\end{equation}
\end{construction}

\begin{proposition}\label{deltalin}
The operation $\varDelta_{(k_1, k_2)}$ is $\SU$-linear.
\end{proposition}

\begin{proof}
By Theorem~\ref{su-linmu} it is enough to check that $\varDelta_{(k_1, k_2)}$ is $SU$-linear on the homotopy groups $\pi_*(\MU)=\MU_*$. Take bordism classes $[M] \in \MU_m$, $[N] \in \MSU_n$. We need to prove that $\varDelta_{(k_1, k_2)}([N \times M])= [N] \cdot \varDelta_{(k_1, k_2)}([M])$.

By definition, 
\[
  \varDelta_{(k_1, k_2)}([N \times M])=\varepsilon D_U \bigl(\bigl(\cf_1(\det \mathcal T (N \times M)) 
  \bigr)^{k_1} \bigl (\cf_1(\overline{\det \mathcal T (N \times M)})\bigr)^{k_2}\bigr),
\]
where $D_U \colon \MU^*(N \times M) \xrightarrow{\simeq} \MU_{n+m-*} (N \times M)$ is the Poincar\'e--Atiyah duality isomorphism and $\varepsilon \colon \MU_*(N \times M) \to \MU_*(\pt)$ is the augmentation.
We have $\det \mathcal T(N\times M)=\det \mathcal T N \otimes \det \mathcal T M = \det \mathcal T M$, as $N$ is an $\SU$-manifold. Hence,
\[
  \varDelta_{(k_1, k_2)}([N \times M])=\varepsilon D_U 
  \bigl(\bigl(\cf_1(\det \mathcal T (M))\bigr)^{k_1} 
  \bigl (\cf_1(\overline{\det \mathcal T (M)})\bigr)^{k_2}\bigr).
\]
Clearly, the submanifold dual to $\bigl(\cf_1(\det \mathcal T ( M))\bigr)^{k_1} \bigl (\cf_1(\overline{\det \mathcal T (M)})\bigr)^{k_2}$ in the product $N \times M$ is $N \times \varDelta_{(k_1, k_2)}([M])$, which implies the statement.
\end{proof}

We denote 
\[
  \varDelta=\varDelta_{(1,1)}, \quad \partial_k= \varDelta_{(k,0)}, \quad 
  \partial=\partial_1, \quad \overline\partial_k=\varDelta_{(0, k)}.
\]
Note that $\partial_0$ is the identity operation. It follows from definition that each operation 
$\varDelta_{(k_1, k_2)}$ can be expressed as a series in $\partial_k$ (as well as in~$\overline \partial_k$).

\begin{lemma}\label{u^i}
Under the isomorphism of Proposition~\ref{bij}, a power series $\sum_{i \ge 0} \lambda_i u^{i+1} \in \widetilde{\MU} \vphantom{MU}^{2-2k}(\C P^{\infty})$ corresponds to the operation $\sum_{i \ge 0} \lambda_i \overline{\partial}_i$, $\lambda_i\in \MU^{-2i-2k}$.
\end{lemma}

\begin{proof}
Indeed, we have $\sum \lambda_i \overline \partial_i(u) = \sum \lambda_i u^{i+1}$ by formula~\eqref{Deltau}. 
\end{proof}

Now we can formulate the main result of the section.

\begin{theorem} \label{su-lin}
Any $\SU$-linear operation $f \in [\MU, \MU]_{\MSU,*}$ can be written uniquely as a series $f=\sum_{i \ge 0} \mu_i \partial_i$, where $\mu_i \in \MU^{-2i-*}$.
\end{theorem}

\begin{proof}
It follows from Proposition~\ref{bij} and Lemma~\ref{u^i} that any $\SU$-linear operation $f$ can be represented as a series 
$\sum \lambda_i \overline{\partial}_i$.  Since $\overline{\partial}_i$ is a series in 
$\partial_i$, we can also write $f$ as a series $\sum \mu_i \partial_i$. The coefficients $\lambda_i$ and $\mu_i$ are determined uniquely by the action of $f$ on the canonical orientation $u \in \widetilde{\MU} \vphantom{MU}^2(\C P^{\infty})$. Namely,  $f(u)=\sum_{i\ge0}\lambda_iu^{i+1}=\sum_{i\ge0}\mu_iu\overline u^i$ by formula~\eqref{Deltau}.
\end{proof}

Theorem~\ref{su-lin} is readily extended to $\SU$-multilinear operations:

\begin{theorem}\label{su-polylin}
Any $\SU$-multilinear operation in complex cobordism can be written uni\-quely as a series 
$\sum \mu_{i_1, \ldots , i_k} \partial_{i_1}\cdots\partial_{i_k}$.
\end{theorem}

This result will be used when we study $\SU$-bilinear multiplications in Subsection~\ref{prsect}.

The product structure (with respect to the composition) of $\SU$-linear operations $f=\sum_{i \ge 0} \beta_i \partial_i$ can be described via the coefficients of the formal group law in complex cobordism, as described next.

For an integer $k\in\Z$, the \emph{$k$-th power} $[k](u)$ in the formal group law $F_U$ is defined inductively:
\[
  [0](u)=0, \quad [k](u)=F\bigl([k-1](u), u\bigr)\text{ for } k > 0, \quad 
  [k](u) = F\bigl([k+1](u), \overline u\bigr)
  \text{ for }k < 0.
\]
For a complex line bundle $\xi$, we have $[k](\cf_1(\xi))=\cf_1(\xi^{\otimes k})$ and 
$[-k](\cf_1(\xi))=\cf_1(\overline{\xi}^{\otimes k})$, $k \ge 0$.

We can now describe the product structure of $\SU$-linear operations.

\begin{theorem}\label{comp}
Given non-negative integers $k,m$, consider the two power series
\[
  (F(u, v))^k = \sum_{i,j\ge0} \alpha_{ij}^{(k)} u^i v^j,
  \quad
 \bigl([1-m](u)\bigr)^k u^m=\sum_{i\ge0} \beta_i^{(k,m)} u^i.
\]
Then
\begin{equation}\label{dkm}
  \partial_k(a\cdot b)=\sum_{i,j} \alpha_{ij}^{(k)} \, \partial_i a \cdot \partial_j b,
  \quad
  \partial_k \partial_m = \sum_{i} \beta_{i}^{(k,m)} \partial_i,
\end{equation}
where $a,b\in[\MU,\MU]_*$ are arbitrary operations, and $a\cdot b\in[\MU\wedge\MU,\MU]_*$ denotes the exterior product.
\end{theorem}

\begin{proof}
By Lemmata~\ref{monic} and~\ref{pmonic} it is enough to verify the identities on the homotopy groups of the $\MU$-spectrum, that is, on elements of
$\MU_*=\pi_*(\MU)$.

Let $a=[M], \, b=[N] \in \MU_*$. We denote $u=\cf_1(\det \mathcal T M)$,  $v = \cf_1(\det \mathcal T N)$. The first identity follows by calculation
\begin{multline*}
  \partial_k([M \times N]) 
  = \varepsilon D_U \bigl( \bigl(\cf_1(\det \mathcal T (M \times N))\bigr)^k \bigr) 
  = \varepsilon D_U \bigl( \bigl(\cf_1(\det \mathcal T M \otimes \det \mathcal T N)\bigr)^k \bigr)\\
  =\varepsilon D_U \bigl( (F(u, v))^k \bigr) 
  = \varepsilon D_U \Bigl(\sum_{i,j} \alpha^{(k)}_{ij} u^i v^j\Bigr) 
  = \sum_{i,j} \alpha^{(k)}_{ij} \partial_i ([M]) \partial_j ([N]),
\end{multline*}
where the last equality follows from the fact that the submanifold $\partial_i (M) \times \partial_j (N)$ is dual to $u^i v^j$ in the product $M \times N$.

For the second identity, let $\partial_m ([M]) = [N]$,  where $i \colon N \hookrightarrow M$ is an embedding. The normal bundle of $N$ is isomorphic to $i^* (\det \mathcal T M)^{\oplus m}$, therefore, 
\[
  \det \mathcal T N \otimes i^*(\det \mathcal T M)^{\otimes m} 
  =\det\bigl(\mathcal T N\oplus i^*(\det \mathcal T M)^{\oplus m}\bigr)
  \cong i^* \det \mathcal T M.
\]
Hence, $\det \mathcal T N = i^*(\overline{\det \mathcal T M})^{\otimes(m-1)}$. Denoting 
$u = \cf_1(\det \mathcal T M)$, we obtain $\cf_1(\det \mathcal T N) = i^* [1-m](u)$. On the other hand, the submanifold $N$ is dual to $u^m$, that is, $i_*[N] = D_U (u^m) = u^m \frown [M]$. Then the following calculation proves the second identity:
\begin{multline*}
  \partial_k[N]=\bigl\langle \bigl(\cf_1(\det \mathcal T N)\bigr)^k , [N] \bigr\rangle 
  = \bigl\langle i^* \bigl([1-m](u)\bigr)^k , [N]  \bigr \rangle 
  = \bigl\langle\bigl([1-m](u)\bigr)^k , i_* [N] \bigr\rangle
  \\=\bigl \langle \bigl([1-m](u)\bigr)^k , u^m \frown [M] \bigr\rangle = 
  \bigl\langle  \bigl([1-m](u)\bigr)^k u^m , [M] \bigr\rangle = 
  \sum_i \beta_i^{(k,m)} \partial_i ([M]).\qedhere
\end{multline*}
\end{proof}

The second relation in~\eqref{dkm} implies $\partial_k\partial=0$, as $\beta_i^{(k,1)}=0$. The identity 
$\partial_k\partial=0$ follows also from the geometric definition of~$\partial_k$.

Relations~\eqref{dkm} can be used to express an arbitrary composition of operations $f=\sum_i\lambda_i\partial_i$ as an operation of the same form.

\section{$c_1$-spherical bordism $W$, projections and multiplications}\label{Wsect}

Here we consider the $c_1$-spherical bordism theory $W$, describe $\SU$-linear multiplications on it and $\SU$-linear projections $\MU \to W$.

\subsection{Definition and the $\MSU$-module structure.}\label{Wdefsect}
The theory $W$ of $c_1$-spherical bordism is defined geometrically as follows (see~\cite[Chapter~VIII]{ston68}). We consider closed manifolds $M$ with a \emph{$c_1$-spherical structure}, which consists of
\begin{itemize}
\item[--] a stably complex structure on the tangent bundle $\mathcal T M$;

\item[--] a \emph{$\C P^1$-reduction} of the determinant bundle, that is, a map $f \colon M \to \C P^1$ and an equivalence $f^*(\eta) \cong \det \mathcal T M$, where $\eta$ is the tautological bundle over $\C P^1$.
\end{itemize}
This is a natural generalisation of an $\SU$-structure, which can be thought of as a ``$\C P^0$-reduction'', that is, a trivialisation of the determinant bundle. The corresponding bordism theory is called \emph{$c_1$-spherical bordism} and is denoted by~$W$.

As in the case of stable complex structures, a $c_1$-spherical complex structure on the stable tangent bundle is equivalent to such a structure on the stable normal bundle. There are forgetful transformations $\MSU\to W \to \MU$.

Homotopically, a $c_1$-spherical structure on a stable complex bundle $\xi \colon M \to \BU$ is defined by a choice of lifting to a map $M \to X$, where $X$ is the (homotopy) pullback:
$$
  \xymatrix{
  & X \ar[r] \ar[d] & \C P^1 \ar[d]^-i \\
  M \ar@{-->}[ur] \ar[r]^{\xi} & \BU \ar[r]^{\det} & \C P^{\infty}
  }
$$
As $\C P^{\infty}$ is a topological abelian group, we can change the pullback square above to the following:
$$
\xymatrix{
& X \ar[r] \ar[d] & \ast \ar[d] \\
M \ar@{-->}[ur] \ar[r]^-{\xi} & \BU \times \C P^1 \ar[r]^-{\det-\,i} & \C P^{\infty}
}
$$
The Thom spectrum corresponding to the map $X \to \BU$ defines the bordism theory of manifolds with a $\C P^1$-reduction of the stable normal bundle, that is, the theory~$W$. We shall also denote this spectrum by~$W$. 

\begin{remark}
In order to get a $\C P^1$-reduction of the stable tangent bundle, we need to replace the inclusion $i \colon \C P^1 \hookrightarrow \C P^{\infty}$ in the pullback squares above by $-i$. 
Replacing the basepoint inclusion $* \to \C P^{\infty}$ by the fibration $S^{\infty} \to \C P^{\infty}$ we obtain the definition given in~\cite[Chapter~8]{ston68}:
$$
  \xymatrix{
  Y \ar[r] \ar[d] & S^{\infty} \ar[d] \\
  \BU \times \C P^1 \ar[r]^-{\det+i} & \C P^{\infty}
  }
$$
\end{remark}

The spectrum $W$ has a natural $\MSU$-module structure. The forgetful morphisms $\MSU \to W \to \MU$ are $\MSU$-module maps. The following description of the $\MSU$-module $W$ is contained implicitly in the works of Atiyah~\cite{ati61}, Conner and Floyd~\cite{co-fl66} and Stong~\cite{ston68} (compare Proposition~\ref{msu-lin}).

\begin{proposition}\label{Wmsu}
There is an equivalence of $\MSU$-modules 
\[
  W \simeq \MSU \wedge \Sigma^{-2}\C P^2.
\]
Under this equivalence, the forgetful maps $\MSU\to W\to\MU$ are identified with the free $\MSU$-module maps 
\[
  \MSU=\MSU \wedge \Sigma^{-2}\C P^1\to\MSU \wedge \Sigma^{-2}\C P^2
  \to\MSU \wedge \Sigma^{-2}\C P^\infty.
\]   
\end{proposition}

\begin{proof}
Consider the commutative diagram
$$
\xymatrix{
\BSU \ar[r] \ar@{=}[d] & X \ar[r]\ar[d]  & \C P^1 \ar[d]^-{-i}\\
\BSU \ar[r] & \BU \ar[r]^-{\det} & \C P^{\infty}
}
$$
where the rows are fibrations and the right square is pullback. The bottom fibration is split (see Proposition~\ref{msu-lin}). Hence, the top fibration is also split, and we obtain a homotopy equivalence $X\simeq\BSU\times\C P^1$. The map $X\to\BU$ is identified with $\id\times(-i)\colon\BSU\times\C P^1\to\BSU\times\C P^\infty$. It induces an equivalence of the corresponding Thom spectra
$W \simeq \MSU \wedge \Sigma^{-2}\C P^2$, as $\C P^2$ is the Thom space of the tautological line bundle over $\C P^1$.
\end{proof}

\begin{proposition}\label{cofiber}
The spectrum $W$ is the cofibre of the multiplication $\Sigma \MSU \xrightarrow{\cdot \theta} \MSU$ by the nontrivial element $\theta \in \MSU_1\cong\Z_2$. The right arrow in the resulting cofibre sequence
\[
  \Sigma \MSU \xrightarrow{\cdot \theta} \MSU \to W
\] 
is the forgetful map.
\end{proposition}

\begin{proof}
The element $\theta$ is given by the Hopf map $S^3 \xrightarrow{\eta} S^2 = \MSU(1)$. The multiplication by $\theta$ is given by the map of spectra
\[
  \MSU \wedge \Sigma^{-2} S^3 
  \xrightarrow{1 \wedge \Sigma^{-2} \eta} \MSU \wedge \Sigma^{-2} S^2.
\]
The cofibre of $S^3 \xrightarrow{\eta} S^2$ is $\C P^2$. Hence, the cofibre of the map above is equivalent to $\MSU \wedge \Sigma^{-2} \C P^2$, which coincides with the spectrum $W$ by Proposition~\ref{Wmsu}.
\end{proof}

\begin{proposition}
The spectra $W$ and $\MSU$ are Bousfield equivalent, that is, $\MSU_*(X)=0$ if and only if $W_*(X)=0$, and a map $X \to Y$ induces an isomorphism 
$\MSU_*(X) \xrightarrow{\cong} \MSU_*(Y)$ if and only if it induces an isomorphism $W_*(X) \xrightarrow{\cong} W_*(Y)$.
\end{proposition}

\begin{proof}
This is a well-known property of the cofibre of a nilpotent map (see, for example, the same argument in \cite[Theorem 8.14]{rav84} for the case of $K$-theory). Consider the homology exact sequence of the cofibration from Proposition~\ref{cofiber}: 
\[
  \cdots \to \MSU_{*-1}(X) \xrightarrow{\cdot \theta} \MSU_*(X) \to W_*(X) \to \MSU_{*-2}(X)   
  \xrightarrow{\cdot \theta} \MSU_{*-1}(X) \to \cdots
\]
Clearly, $\MSU_*(X)=0$ implies $W_*(X)=0$. For the opposite direction, if $W_*(X)=0$, then $\MSU_{*-1}(X) \xrightarrow{\cdot \theta} \MSU_*(X)$ is an isomorphism. As $\theta^3 \in MSU_3 =0$ (see, for instance, \cite[Example 5.7]{c-l-p19}), we obtain that $\MSU_*(X)=0$.

The second assertion (about an isomorphism in homology) follows from the first one by considering the homology exact sequence of the map~$X\to Y$.
\end{proof}

\begin{proposition}
The integral homology $H_*(W)$ is concentrated in even dimensions, and there are the short exact sequences 
\[
  0 \to H_{2k}(\MSU) \to H_{2k}(W) \to H_{2k-2}(\MSU) \to 0
\] 
In particular, the homology $H_*(W)$ has no torsion.
\end{proposition}

\begin{proof}
Consider the integral homology exact sequence of the cofibration from Proposition~\ref{cofiber}: 
\[
  \ldots \to H_{2k-1}(\MSU) \to H_{2k}(\MSU) \to H_{2k}(W) \to H_{2k-2}
  (\MSU) \to H_{2k-1}(\MSU) \to \ldots
\]
As the homology $H_*(\MSU)$ is concentrated in even dimensions and has no torsion, the same is true for~$W$.
\end{proof}

Recall from Section~\ref{linsect} that the operation $\partial\colon\MU\to\Sigma^2\MU$ assigns to a bordism class $[M]\in\MU_*$ the bordism class of a submanifold dual to~$c_1(M)$.

\begin{proposition}\label{conn}
The composite
\[
  W\xrightarrow{\partial'} \Sigma^2\MSU\to\Sigma^2\MU
\]
of the connecting map in the cofibre sequence of Proposition~\ref{cofiber} and the forgetful map coincides with $-\partial\colon W\to\Sigma^2\MU$.
\end{proposition}

\begin{proof}
In view of Lemma~\ref{monic}, it is enough to verify the claim on the homotopy groups of spectra, that is, to check that $W_{2n}\xrightarrow{\partial'} \MSU_{2n-2}\to \MU_{2n-2}$ coincides with $-\partial$. This is proved in~\cite[(17.3)]{co-fl66} (Conner and Floyd define $W_*$ as 
$\ker\varDelta$, see Proposition~\ref{kerDelta} below).
\end{proof}

Combining Proposition~\ref{cofiber} and Proposition~\ref{conn} we obtain the exact sequence
\[
\cdots \to \Sigma \MSU \xrightarrow{\cdot \theta} \MSU \to W \xrightarrow{\partial'} \Sigma^2 \MSU \xrightarrow{\cdot \theta} \Sigma \MSU \to \cdots
\]
of Conner and Floyd \cite{co-fl66}. For the homotopy groups of spectra we obtain a 5-term exact sequence~\cite[(18.1)]{co-fl66}
\[
  0\longrightarrow \MSU_{2n-1}\stackrel{\cdot \theta}\longrightarrow
  \MSU_{2n}\longrightarrow W_{2n}
  \stackrel{\partial'}\longrightarrow \MSU_{2n-2}
  \stackrel{\cdot \theta}\longrightarrow \MSU_{2n-1}\longrightarrow0.
\]

\subsection{Relation to the operation $\varDelta$}\label{Deltasect}

Recall that $W_*=\pi_*(W)$ denotes the homotopy groups (coefficients) of the spectrum~$W$.

\begin{construction}[{\cite[Chapter VIII]{ston68}}] Define a homomorphism $\pi_0 \colon \MU_* \to W_*$ sending a bordism class $[M]$ to the class of the submanifold $N \subset \C P^1 \times M$ dual to $\overline\eta \otimes \det \mathcal T M$. We have $\det \mathcal T N\cong i^*\overline\eta$, where $i$ is the embedding $N \hookrightarrow \C P^1 \times M$, so $N$ has a natural $c_1$-spherical stably complex structure.
\end{construction}

\begin{proposition}[{\cite[Chapter VIII]{ston68}}]
The composite $W_* \to \MU_* \xrightarrow{\pi_0} W_*$ is the identity map.
In particular, the image of the forgetful homomorphism $W_* \to \MU_*$ is a direct summand of~$\MU_*$.
\end{proposition}

\begin{corollary}
The groups $W_*$ are concentrated in even dimensions and are torsion-free.
\end{corollary}

We can also view $\pi_0$ as an idempotent homomorphism of abelian groups 
$\MU_* \to \MU_*$ and refer to it as the \emph{Stong projection}.

\begin{proposition} \label{pi}
For any $a \in \MU_*$ we have
\begin{equation}\label{stong}
  \pi_0 (a) = a + \sum_{k \ge 2} \alpha_{1k} \partial_k a,
\end{equation}
where $\alpha_{1k}$ are the coefficients of the formal group law $F_U$ in complex cobordism~\eqref{fglcc}. Furthermore,
$$\partial \pi_0 = \pi_0 \partial = \partial.$$
\end{proposition}

\begin{remark}
By Lemma~\ref{monic}, formula~\eqref{stong} extends $\pi_0$ uniquely to a cohomological operation from $[\MU, \MU]$.
\end{remark}

\begin{proof}[Proof of Proposition \ref{pi}]
Let $a = [M]$. By the definition of~$\pi_0$,
\[
  \pi_0 (a) = \varepsilon D_U(\cf_1 (\overline{\eta} \otimes \det \mathcal T M)),
\] 
where $D_U \colon \MU^2(\C P^1 \times M^n) \xrightarrow{\cong} \MU_n(\C P^1 \times M^n)$ is the Poincar\'e--Atiyah duality isomorphism and $\varepsilon \colon \MU_*(X) \to \MU_*(\pt)$ is the augmentation.

Let $u = \cf_1(\overline{\eta})$, $v = \cf_1(\det\mathcal T M) \in \MU^2(\C P^1 \times M)$. Then
\begin{multline*}
  \varepsilon D_U(\cf_1 (\overline{\eta} \otimes \det \mathcal T M))= \varepsilon D_U (F(u, v)) 
  = \varepsilon D_U(u) + \varepsilon D_U (v)+\sum_{i, j \ge 1} \alpha_{ij} \varepsilon D_U(u^i v^j)
  \\ = [M] + [\C P^1] \partial [M] + \sum_{j \ge 1} \alpha_{1j} \partial_j [M],
\end{multline*}
where we used the identities $u^2=0$, $\varepsilon D_U (u v^j) = \partial_j [M]$ and $\varepsilon D_U(v)=[\C P^1] \partial[M]$. Formula~\eqref{stong} follows by noting that $\alpha_{11} = -[\C P^1]$.

The identity $\pi_0 \partial = \partial$ is obtained by applying~\eqref{stong} to $\partial a$ and using the identity $\partial_k \partial = 0$.

It remains to prove $\partial \pi_0 = \partial$. 
Let $\pi_0[M] = [N]$. We need to show that $\partial [N] = \partial [M]$. We have $\det \mathcal T N = i^*\overline \eta$, where $i \colon N \hookrightarrow \C P^1 \times M$, and 
\[
  i_*[N] = D_U(\cf_1(\overline \eta \otimes \det(\mathcal T M))) = D_U(F_U(u, v)) = 
  F_U(u, v)\frown [M \times \C P^1].
\]
Then the required identity follows by calculation:
\begin{multline*}
  \partial [N] = \varepsilon D_U (\cf_1 (\det \mathcal T N)) = \varepsilon D_U (i^*u) 
  = \langle i^*u, [N] \rangle = \langle u, i_*[N] \rangle\\ 
  = \langle u, F_U(u, v)\frown [M \times \C P^1]) =\langle u F_U(u, v), [M \times \C P^1] \rangle 
  = \langle uv, [M \times \C P^1] \rangle = \partial [M].\qedhere
\end{multline*}
\end{proof}

\begin{remark}
Formula~\eqref{stong} implies that the projection $\pi_0$ is $\SU$-linear, which is also clear from its geometric definition.
\end{remark}

\begin{proposition}\label{kerDelta}
The image of the forgetful homomorphism $W_* \to \MU_*$ coincides with $\ker \varDelta$.
\end{proposition}

\begin{proof}
The operation $\varDelta$ sends a bordism class $[M]$ to the class of the submanifold $[N]$ dual to $\det \mathcal T M \oplus \overline{\det \mathcal T M}$. For a $c_1$-spherical manifold~$M$, the bundle $\det \mathcal T M$ is induced from the tautological bundle $\eta$ over $\C P^1$. Since $\eta \oplus \overline \eta$ is trivial over $\C P^1$, the operation $\varDelta$ vanishes on the image of $W_*$.

Conversely, let $a \in \ker \varDelta$. By~\cite[Corollary~6.4]{c-l-p19}, $\partial_k a = 0$ for $k \geq 2$. Then~\eqref{stong} implies $\pi_0(a)=a$, so $a$ is in the image of the forgetful homomorphism $W_* \to \MU_*$.
\end{proof}

\begin{remark}
The coefficient group $W_*$ was originally introduced by Conner and Floyd~\cite{co-fl66} precisely as $\ker \varDelta$.
\end{remark}

There is the following homotopical description of the spectrum $W$.

\begin{proposition}\label{fiber}
The spectrum $W$ is the fibre of the map $\MU \xrightarrow{\varDelta} \Sigma^4 \MU$.
\end{proposition}

\begin{proof}
Denote the fibre by $F$. We have a long exact sequence of homotopy groups 
\[
  \cdots \to \pi_{*-3}(\MU) \to \pi_*(F) \to \pi_* (\MU) \xrightarrow{\varDelta} \pi_{*-4}(\MU) 
  \to \pi_{*-1}(F) \to \cdots
\]
The operation $\varDelta$ has a right inverse (see \cite[Lemma 4.3]{c-l-p19} and Example~\ref{CFproj} below), so it is surjective. Therefore, the long exact sequence above splits: 
$$
  0 \to \pi_*(F) \to \pi_* (\MU) \xrightarrow{\varDelta} \pi_{*-4}(\MU) \to 0
$$

Proposition \ref{kerDelta} implies that there are similar short exact sequences for $W_*$:
$$
  0 \to \pi_*(W) \to \pi_* (\MU) \xrightarrow{\varDelta} \pi_{*-4}(\MU) \to 0
$$
This short exact sequence together with the fact that  $H_*(W)$ is torsion-free and Lemma~\ref{monic} implies that the composite $W \to \MU \xrightarrow{\varDelta} \Sigma^4 \MU$ is homotopic to zero. Hence, there is a morphism $W \to F$ making the following diagram commutative
$$
  \xymatrix{
  0 \ar[r] & \pi_*(W) \ar[r] \ar[d] & \pi_* (\MU) \ar[r]^{\varDelta} \ar@{=}[d] & \pi_{*-4}(\MU)  
  \ar[r] \ar@{=}[d] & 0\\
  0 \ar[r] & \pi_*(F) \ar[r] & \pi_* (\MU) \ar[r]^{\varDelta} & \pi_{*-4}(\MU) \ar[r] & 0
}
$$
It follows that the map $W\to F$ induces an isomorphism of homotopy groups, and therefore it is an equivalence of the spectra.
\end{proof}

\begin{proposition}\label{WkerD}
For any space (or spectrum) $X$, the forgetful map $W_*(X) \to \MU_*(X)$ is injective and its image coincides with $\ker \varDelta$. The same holds for $W^*(X)$.
\end{proposition}

\begin{proof}
By Proposition~\ref{fiber} we obtain a long exact sequence
$$\cdots \to W_*(X) \to \MU_*(X) \xrightarrow{\varDelta} \MU_{*-4}(X) \to \cdots$$
Since $\varDelta$ has a right inverse, this long exact sequence splits into short ones
\[
  \qquad\qquad0 \to W_*(X) \to \MU_*(X) \xrightarrow{\varDelta} 
  \MU_{*-4}(X) \to 0.\qquad\qquad\qedhere
\]  
\end{proof}

Propositions~\ref{kerDelta} and~\ref{WkerD} were proved in~\cite[Chapter VIII]{ston68} by geometric methods.

\begin{remark} Proposition \ref{fiber} also implies that the Stong projection $\pi_0 \in [\MU, \MU]$ lifts uniquely to an operation $\pi_0 \in [\MU, W]$.
\end{remark}

\subsection{$\SU$-linear projections onto $W$ and $\SU$-linear multiplications}\label{prsect}
A morphism $\MU\to W$ is called a \emph{projection onto $W$} if it is identical on~$W$, where $W$ is viewed as a submodule of~$\MU$ via the forgetful morphism~$W\to\MU$. We often view such projections as idempotent morphisms $\MU\to\MU$ with image~$W$. An example is given by the Stong projection $\pi_0\colon\MU\to W$.

Any such projection maps $\MU_*(X)$ onto the direct summand $W_*(X) = \ker \varDelta$, and similarly for $W^*(X)$. Moreover, Proposition~\ref{fiber} implies that any projection $\MU\to W$ defines a splitting of the complex cobordism spectrum $\MU \simeq W \vee \Sigma^4 \MU$, and the fibre sequence from Proposition \ref{fiber} also splits.

The Stong projection is $SU$-linear, and therefore can be written as a series in~$\partial_k$. The coefficients of this series are given by~\eqref{pi}. 
More generally, we have

\begin{proposition}\label{su-pr}
Any $\SU$-linear projection $\MU\to W$  has the form $\pi = 1 + \sum_{i \ge 2} \lambda_i \partial_i$ with $\lambda_i\in \MU^{-2i}$.
\end{proposition}
\begin{proof}
We write $\pi=\sum_{i\ge0}\lambda_i\partial_i$ according to Theorem~\ref{su-lin}. Then $\pi(1)=1$ and $\pi([\C P^1])=[\C P^1]$, because $[\C P^1]\in W_2$. Since $\partial[\C P^1]=2$ and $\partial_i[\C P^1]=0$ for $i\ge2$, we obtain $\lambda_0=1$ and $\lambda_1=0$, as needed.
\end{proof}

\begin{theorem}\label{pr1}
Let $\pi \colon \MU \to W$ be a projection onto $W$. Then any other projection $\MU \to W$ has the form $\pi(1+f \varDelta)$ for some operation $f \in [\MU, \Sigma^{-4} \MU]$. Furthermore, if $\pi$ is $\SU$-linear, then any other $\SU$-linear projection has the form $\pi(1+f \varDelta)$ with $SU$-linear~$f$.
\end{theorem}

\begin{proof}
The (split) fibre sequence of Proposition~\ref{fiber} gives the exact sequence 
\[
  \cdots \gets [\Sigma^3 \MU, W] \gets [W, W] \gets [\MU, W] \gets [\Sigma^4 \MU, W] \gets \cdots
\]
A projection $\MU \to W$ is an element of $[\MU, W]$ that maps to the identity in $[W, W]$. Such a projection exists because $[\Sigma^3 \MU, W] = 0$ (the homotopy groups $W_*$ are concentrated in even dimensions). Furthermore, any two projections $\MU \to W$ differ by the image of an element from $[\Sigma^4 \MU, W]$. That is, any projection has the form $\pi + g \varDelta$, where $g \in [\Sigma^4 \MU, W]$. It remains to note that any $g \in [\Sigma^4 \MU, W]$ can be written as $\pi f$ for $f \in [\Sigma^4 \MU, \MU]$. This proves the first statement.

Now suppose $\pi$ is an $SU$-linear projection. Then $SU$-linear operations $f$ give $SU$-linear projections $\pi(1+f \varDelta)$. Conversely, if a projection $\pi (1 + f \varDelta)$ is $\SU$-linear, then the operation $\pi f \varDelta$ is also $SU$-linear. Denote by $f' \in [\Sigma^4 \MU, \MU]$ the composite  of $\pi f \in [\Sigma^4 \MU, W]$ and the forgetful map $W \to \MU$. Then the operation $f' \varDelta$ is $\SU$-linear. Since $\varDelta$ has a right inverse, $f'$ is also $SU$-linear. Now $\pi f' \varDelta = \pi f \varDelta$, so we can replace $f$ in $\pi (1 + f \varDelta)$ by an $SU$-linear operation~$f'$.
\end{proof}

\begin{lemma}\label{gDelta}
The following three groups of $\SU$-linear operations coincide:
\begin{itemize}
\item[(1)] $\SU$-linear operations vanishing on~$W$;

\item[(2)] operations of the form $g \varDelta$ with $\SU$-linear $g$;

\item[(3)] operations $\sum_{i \ge 2} \lambda_i \partial_i$, $\lambda_i\in MU_*$.
\end{itemize}
\end{lemma}

\begin{proof}
Any operation $g \varDelta$ vanishes on $W$ by Proposition~\ref{fiber}. On the other hand, $\sum_{i \ge 2} \lambda_i \partial_i$ vanishes on $W$ by~\cite[Corollary 6.4]{c-l-p19}.

Conversely, if an operation vanishes on $W$, then it has the form $g \varDelta$ by Proposition~\ref{fiber}. If $g \varDelta$ is $\SU$-linear, then $g$ is also $\SU$-linear, as $\varDelta$ has a right inverse.

Finally, by Theorem \ref{su-lin}, any $\SU$-linear operation has the form $\sum_{i \geq 0} \lambda_i \partial_i$. If it vanishes on $W$, then calculating at $1 \in \pi_0(W)$ and $[\C P^1] \in \pi_2(W)$ we obtain $\lambda_0=\lambda_1=0$.
\end{proof}

\begin{theorem}\label{pr2}
Any projection $\MU\to W$ has the form $1-f \varDelta$, where $f$ is an arbitrary operation satisfying 
$\varDelta f =1$. Furthermore, different projections correspond to different~$f$, and $\SU$-linear projections correspond to $\SU$-linear $f$.
\end{theorem}

\begin{proof}
The split fibration sequence of Proposition~\ref{fiber} gives rise to a short exact sequence
\[
  0 \gets [W, \MU] \gets [\MU, \MU] \gets [\Sigma^4 \MU, \MU] \gets 0.
\]

Let $p\in[\MU, \MU]$ be a projection onto~$W$. Since it is identical on $W$, it maps to the forgetful morphism $W \to \MU$. The identity $1\in [\MU, \MU]$ also maps to the forgetful morphism, so we obtain 
$1-p= f \varDelta$ for some $f\in [\Sigma^4 \MU, \MU]$, and different $f$ correspond to different~$p$. Hence, $p=1-f\varDelta$. It is a projection onto $W$ if and only if $\varDelta (1-f \varDelta) = 0$. 
As $\varDelta$ has a right inverse, we obtain $1-\varDelta f = 0$. Conversely, the latter condition implies that $\varDelta (1-f \varDelta) = 0$.
The existence of a right inverse for $\varDelta$ also implies that $p$ is $\SU$-linear if and only if so is~$f$.
\end{proof}

An $SU$-linear projection $\pi \colon \MU \to W$ defines an $\SU$-bilinear multiplication on $W$ by the formula 
\begin{equation}\label{multpi}
  W \wedge W \to \MU \wedge \MU \xrightarrow{m_{\MU}} \MU \xrightarrow{\pi} W.
\end{equation}
As $\pi$ is a projection, this multiplication has a unit, obtained from the unit of 
$\MSU$ by the forgetful morphism.

Given elements $a,b\in W_*$, we denote by $ab$ the product of their images in 
$\MU_*$ under the forgetful morphism. 

\begin{proposition}\label{multpr}
The multiplication~\eqref{multpi} corresponding to an $\SU$-linear projection 
$\pi=1+\sum_{i \ge 2} \lambda_i \partial_i$ is given by
\[
  a*b=ab+2\lambda_2 \partial a \partial b,
\]
This formula can be understood as an identity of the operations from $[W\wedge W, W]_*$, or as an identity for cohomology classes $a,b\in[E,W]_*$, where~$E$ is an arbitrary spectrum. 

In particular, the multiplication defined by the Stong projection $\pi_0=1 + \sum_{k \ge 2} \alpha_{1k} \partial_k$ is given by
\[
  a*b = ab + 2[V]\partial a \partial b.
\]
where $[V]=\alpha_{12}\in \MU_4$ is the cobordism class $[\C P^1]^2 - [\C P^2]$.
\end{proposition}
\begin{proof}
It is enough to verify the identity on the elements of $W_*=[S,W]_*$. We use the formula from Theorem~\ref{comp} and the fact that $\partial_i$ vanishes on 
$W_*$ for $i \ge 2$ (see \cite[Corollary~6.4]{c-l-p19}):
\[
  a*b=\pi(ab)=ab+\lambda_2\partial_2(ab)+\sum_{i\ge3}\lambda_i\partial_i(ab)=
  ab+\lambda_2\alpha^{(2)}_{11}\partial a\partial b=ab+2\lambda_2\partial a\partial b.\qedhere
\]
\end{proof}

\begin{lemma}[{see~\cite[Lemma~6.5]{c-l-p19}}]\label{Deltaab}
For any elements $a,b\in W_*$,
\begin{gather*}
  \partial (ab) = a \, \partial b + \partial a\,b - [\C P^1] \partial a \, \partial b,\\
  \varDelta(a b)=-2\partial a\,\partial b.
\end{gather*}
\end{lemma}

Here is an alternative way to describe multiplications on $W$ corresponding to $\SU$-linear projections.

\begin{proposition}\label{multpr1}
The multiplication~\eqref{multpi} corresponding to an $\SU$-linear projection 
$\pi$ is given by the formula 
\[
  a*b = ab + 2 ([V] - \omega) \partial a \partial b,
\]
where $[V]=\alpha_{12}=[\C P^1]^2 - [\C P^2]$ and $\omega=\pi[V]\in W_4$. Furthermore, any element of $W_4$ can be taken as $\omega$ for some~$\pi$.
\end{proposition}

\begin{proof}
By Theorem~\ref{pr1}, $\pi = \pi_0 + \pi_0 f \varDelta$ for an $SU$-linear $f \in [\MU, \Sigma^{-4} \MU]$. Then we use the formulae from Proposition~\ref{multpr} and Lemma~\ref{Deltaab} to calculate
\[
  a * b = \pi_0(ab) + \pi_0 f \varDelta (ab) =  ab + 2[V] \partial a \partial b - 2 \pi_0 f (1) 
  \partial a \partial b.
\]
In the last identity we used the fact that $\pi_0f$ is $SU$-linear. Clearly, any element $\omega\in W_4$ can be taken as~$\pi_0 f(1)$, proving the identity. Now we have 
\[
  a*b = \pi(a*b)= \pi(ab) + 2\pi([V] - \omega)\partial a \partial b 
  = a*b + 2\pi([V] - \omega)\partial a \partial b.
\]
Therefore, $\pi([V] - \omega)=0$ and $\pi [V] = \pi(\omega) = \omega$.
\end{proof}

\begin{example}\label{CFproj}
Conner and Floyd~\cite{co-fl66} defined geometrically a right inverse to the operation $\varDelta$ on the complex bordism groups~$\MU_*$. Novikov~\cite{novi67} extended it to a cohomological operation $\varPsi \in[\Sigma^4 \MU, \MU]$ (see~\cite[Construction 4.2]{c-l-p19}). We therefore obtain an example of a projection of the form described in Theorem~\ref{pr2}, the \emph{Conner--Floyd projection} $1 - \varPsi \varDelta$. As observed in~\cite{c-l-p19} this projection is different from the Stong projection~$\pi_0$, although the two projections define the same multiplication on~$W$. This reflects the fact that both Stong and Conner--Floyd projections have the same coefficient of $\partial_2$ in their expansions $1+\sum_{i \ge 2} \lambda_i \partial_i$.
\end{example}

\begin{theorem}\label{multgen}
Any $\SU$-bilinear multiplication on~$W$ with the standard unit (obtained by the forgetful map from the unit of $\MSU$) has the form 
\[
  a*b=ab + (2 [V] - \omega) \partial a \partial b
\] 
for  $\omega \in W_4$. Any such multiplication is associative and commutative. Furthermore, the multiplications obtained from $\SU$-linear projections are those with $\omega=2\widetilde \omega$, 
$\widetilde\omega \in W_4$.
\end{theorem}

\begin{proof}
Let $m(a, b)$ be an arbitrary $\SU$-bilinear operation on $W$. Given an $\SU$-linear projection $\pi\colon\MU \to W$, we obtain an $\SU$-bilinear operation on $\MU$ by composing $m(\pi(x), \pi(y))$ with the forgetful homomorphism $W \to \MU$. By Theorem~\ref{su-polylin}, such an operation can be written as a series in products of~$\partial_i$. Restricting back to $W$ and using the fact that $\partial_i$ vanishes on $W$ for $i\ge2$, we obtain 
\[
  m(a, b) = \alpha\, ab+\beta\,\partial a\,b+\gamma\, a\,\partial b+\delta\,\partial a\,\partial b.
\]
The unit identity $m(a, 1) = a$ implies $\alpha a+\beta\partial a =a$. Substituting $a=1$ and $a=[\C P^1]$ we obtain $\alpha = 1$, $\beta= 0$. Similarly, $\gamma=0$. Finally, the multiplication takes values in $W$ if and only if 
\[
  0=\varDelta m(a, b) = \varDelta (ab + \delta\,\partial a\, \partial b) = 
  -2 \partial a\, \partial b + \varDelta\delta\, \partial a\, \partial b.
\]
Hence, $\varDelta\delta=2$. Since $\varDelta [V] = 1$, this is equivalent to $\delta = 2[V] - \omega$, 
$\omega \in W_4$.

The multiplication $a*b$ is clearly commutative. For the associativity, we have
\[
  (a*b)*c = (ab +\delta \, \partial a \, \partial b)*c = (ab +\delta \, \partial a \, \partial b)c + \delta \,  
  \partial (ab +\delta \, \partial a \, \partial b) \, \partial c.
\]
The operation $\partial$ is $\SU$-linear and vanishes on $MU_4$, hence, 
$\partial (\delta \, \partial a \, \partial b) = \partial\delta \, \partial a \, \partial b = 0$. We also have 
$\partial (ab) = a \, \partial b + b \, \partial a - [\C P^1] \partial a \, \partial b$ by Lemma~\ref{Deltaab}. Therefore,
\[
  (a*b)*c = abc + \delta \, \partial a \, \partial b \, c + \delta \,  a \, \partial b \, \partial c 
  + \delta \, b \, \partial a \, \partial c - \delta [\C P^1] \partial a \, \partial b \, \partial c=a*(b*c).
\]
Finally, Proposition~\ref{multpr1} implies that a multiplication obtained from an 
$\SU$-linear projection has $\omega$ divisible by~$2$.
\end{proof}

We refer to~\cite{b-b-n-y00} for a general algebraic approach to multiplications in complex cobordism resulting from projections.

\section{Complex orientations of $W$ and formal group laws}\label{fgsect}

This last section is motivated by the work~\cite{buch72} of Buchstaber. We start by an observation that the $c_1$-spherical bordism theory~$W$ is complex oriented for any multiplication~\eqref{multpi}. Furthermore, any complex orientation of $W$ is obtained from a complex orientation of $\MU$ by applying an $\SU$-linear projection~$\pi$. A complex orientation $w$ of $W$ defines a formal group law 
$F_W(u,v)$ in the theory~$W$. Unlike the situation with complex bordism, the coefficients of the formal group law $F_W$ \emph{do not} generate the coefficient ring $W_*$ of the theory~$W$, for any choice of $w$ and~$\pi$. This result is stated in~\cite{buch72} with a short sketch of proof (more details are provided in the case of the Stong projection~$\pi_0$). We give a complete proof using the technique developed in the previous sections.

By Theorem~\ref{multgen}, an $\SU$-bilinear multiplication on $W$ is given by the formula
\begin{equation}\label{mult}
  a*b = ab + \delta \partial a \partial b,
\end{equation}
where $\delta = 2[V] - \omega$, $\omega\in W_4$. As $\partial\delta=0$, we also have
\begin{equation}\label{partial}
  \partial (a*b) = \partial (ab) = a\, \partial b + b\, \partial a - [\C P^1] \partial a\, \partial b.
\end{equation}

In what follows we fix an $\SU$-bilinear multiplication on~$W$.

\begin{proposition}\label{worie}
The theory $W$ is complex orientable. For any $\SU$-linear projection $\pi \colon \MU \to W$ and a complex orientation $\widetilde u \in \widetilde {\MU} \vphantom{MU}^2(\C P^{\infty})$, the element $\pi(\widetilde u) \in \widetilde W^2(\C P^{\infty})$ is a complex orientation for~$W$. Furthermore, for any complex orientation $w$ of $W$ and any $\SU$-linear projection $\pi \colon \MU \to W$, there exists a complex orientation $\widetilde u \in \widetilde {\MU} \vphantom{MU}^2(\C P^{\infty})$ such that $w=\pi(\widetilde u)$.
\end{proposition}

\begin{proof}
Let $u \in \widetilde {\MU} \vphantom{MU}^2(\C P^{\infty})$ be the canonical orientation of complex cobordism. Then $\widetilde u |_{\C P^1} = u|_{\C P^1}$. Hence, $\pi(\widetilde u)|_{\C P^1} = \pi(u)|_{\C P^1} = u|_{\C P^1}$ by Proposition~\ref{su-pr}, as $\partial_i u|_{\C P^1} = 0$ for $i \ge 1$. This implies that 
$\pi(\widetilde u)$ is a complex orientation of~$W$.

Conversely, given an orientation $w \in \widetilde W^2(\C P^{\infty})$, its image under the forgetful map $\widetilde W^2(\C P^{\infty})\to \widetilde {\MU} \vphantom{MU}^2(\C P^{\infty})$ is a complex orientation 
$\widetilde u$ of $\MU$. Therefore, $w=\pi(\widetilde u)$ for any $\SU$-linear projection $\pi \colon \MU \to W$.
\end{proof}

\begin{remark}
As one can see from the proof, the statement holds because a complex orientation of a cohomology theory is defined via the unit of the theory. The forgetful map $W \to \MU$ and projections $\MU \to W$ preserve the standard unit (from $\MSU$) and therefore send orientations to orientations.
\end{remark}

A complex orientation $w\in\widetilde W^2(\C P^\infty)$ defines a formal group law with coefficients in $W^*$, which we denote by $F_W(u,v)$ (it depends on both the multiplication and $w$, but we do not reflect this in the notation). For example, we may take $w=\pi_0(u)$, where $u \in \widetilde{\MU} \vphantom{MU}^2(\C P^{\infty})$ is the canonical orientation and $\pi_0$ is the Stong projection.
The formal group law $F_W$ is classified by a multiplicative transformation 
$\psi\colon\MU\to W$ sending $u$ to~$w$. However, even in the case 
$w=\pi_0(u)$, the transformation $\psi$ does not coincide with the projection 
$\pi_0\colon\MU\to W$, as the latter is \emph{not} multiplicative. 
To study the formal group law $F_W$ we map $W$ further to a one-parameter extension of the $U$-theory, as described next.

\begin{construction}\label{Gamma}
Following \cite{buch72}, consider the multiplicative cohomology theory  $\varGamma$ defined by
\[
  \varGamma^*(X) = \MU^*(X)[t]/(t^2=-[\C P^1]t+\delta)
\] 
for any CW-complex $X$. Additively, $\varGamma^*(X)$ is a free $\MU^*(X)$-module with basis $\{1,t\}$, and the multiplication is defined by the relation $t^2=-[\C P^1]t+\delta$.

There is a natural transformation $\varphi \colon W \to \varGamma$, given by $\varphi(x) = x + t \partial x$.
\end{construction}

\begin{proposition}[{\cite[Lemma 2]{buch72}}]
The transformation $\varphi \colon W \to \varGamma$ is multiplicative.
\end{proposition}

\begin{proof}
Given $a,b\in W_*$, we use~\eqref{mult} and~\eqref{partial} to calculate
\begin{multline*}
(a+t \partial a) (b + t \partial b) = ab + t (a \partial b + b \partial a) + t^2 \partial a \partial b = \\
= a*b + t \partial (a*b) +(t^2 + t [\C P^1]-\delta) \partial a \partial b = a*b + t \partial (a*b).\qedhere
\end{multline*}
\end{proof}

There is a canonical orientation of the theory $\varGamma$ given by the image of the canonical orientation $u \in \widetilde{\MU} \vphantom{MU}^2(\C P^{\infty})$ under the natural inclusion $\MU \hookrightarrow \varGamma$. This orientation of $\varGamma$ will be also denoted by~$u$.

The map $\varphi$ sends the orientation $w$ to an orientation $\varphi(w)$ of the theory $\varGamma$. Hence, $\varphi (w)$ can be written as a power series $\gamma(u)$ in $u$ with coefficients in $\varGamma^*=\varGamma^*(pt)$.

\begin{proposition}[{\cite[Lemma 3]{buch72}}]\label{fg}
We have 
\[
  \varphi_* F_W (u, v) = \gamma F_U (\gamma^{-1}(u), \gamma^{-1} (v)),
\]
where $F_U(u, v)$ is the formal group law in complex cobordism, considered as a formal group law over 
$\varGamma^*$ via the natural inclusion $\MU \hookrightarrow \varGamma$.
\end{proposition}

\begin{proof}
Since $\varphi$ is multiplicative, the formal group law corresponding to the orientation $\varphi (w)$ is 
$\varphi_* F_W$. Similarly, since the inclusion $\MU \hookrightarrow \varGamma$ is also multiplicative, the formal group law corresponding to the orientation $u$ is~$F_U$, regarded as a formal group law 
over~$\varGamma^*$. Now $\varphi(w)=\gamma(u)$ implies the required identity.
\end{proof}

\begin{construction}
Denote by $J=\MU^{<0} \subset \MU^*$ the ideal of elements of non-zero degree. Then $J^2$ is the ideal of decomposables in~$\MU^*$. Clearly, $J^2 + t J$ is an ideal in~$\varGamma^*$.

Consider the quotient ring $R = \varGamma^*/(J^2+tJ)$. As a graded abelian group, 
$R=(\MU^*\!/J^2) \oplus \Z\langle t \rangle$, $\deg t = -2$. The multiplication in $R$ is defined by  $ab=0$, $at=0$ for $a,b\in J/J^2$ and $t^2 = \delta$, so that $t^3=0$. This implies $R^{<-2} R^{<0} = 0$.
\end{construction}

We write
\[  
  F_W (u, v) = u+v+\sum_{i\ge1,\,j\ge1} \omega_{ij} u^i v^j.
\]
In order to compare the subring generated by the coefficients $\omega_{ij}$ with the whole ring 
$W^*$ we need to identify the characteristic $s_k$-numbers of~$\omega_{ij}$. (Recall that $s_k$ is the Chern characteristic number corresponding to the symmetric polynomial $t_1^k+\cdots+t_n^k$ in Chern roots; it vanishes on decomposables $J^2\subset\MU^*$.)

We shall calculate the formal group law $\varphi_* F_W(u, v) = u + v + \sum (\omega_{ij} + t \partial \omega_{ij}) u^i v^j$ over the ring $R$ (that is, reducing the coefficents mod $J^2 + t J$), using the formula from Proposition~\ref{fg}. As the $s_k$-numbers vanish on $J^2$, in this way we obtain information about the $s_k$-numbers of the coefficients of the formal group law~$F_W$.

\begin{lemma}\label{gamma}
The following identity holds in $\varGamma^*$:
\[
  \gamma(u)= u - (\lambda + (2\ell+1)t) u^2 + \sum_{i \ge 2} \gamma_{i+1} u^{i+1} \mod J^2+tJ,
\]
where $\lambda \in\MU^{-2} = W^{-2}$, $2\ell=\partial\lambda$,  $\ell \in \Z$, and 
$\gamma_{i+1} =  (-1)^i\alpha_{1i} + \omega_i$, $\omega_i \in W^{-2i}$.
Furthermore, any $\lambda$ and $\omega_i$ can be obtained from some orientation 
$w\in \widetilde W^2(\C P^{\infty})$.
\end{lemma}

\begin{proof}
By Proposition~\ref{worie}, every complex orientation $w$ has the form $\pi_0(\widetilde u)$ for some orientation $\widetilde u$ in $\MU$. Writing  $\widetilde u\in\widetilde{\MU} \vphantom{MU}^2(\C P^\infty)$ as a power series $f(u)$ in the standard orientation~$u$ and observing that $u^{i+1} = \overline \partial_i u$, we obtain
\begin{equation}\label{wtu}
  \widetilde u=f(u)=u+\sum_{i \ge 1} \lambda_i u^{i+1}=
  \bigl(1 + \sum_{i \ge 1} \lambda_i \overline \partial_i\bigr)u=(1+\lambda\partial+g\varDelta)u.
\end{equation}
In the last identity we used Theorem~\ref{su-lin} and Lemma~\ref{gDelta} to write the $\SU$-linear operation $f = 1 + \sum_{i \ge 1} \lambda_i \overline \partial_i$ as $1+\lambda\partial+g\varDelta$, for a certain $\SU$-linear operation~$g$. Note that any 
$\lambda$ and $g$ can be obtained from an orientation $\widetilde u$.

Now we calculate:
\begin{multline}\label{gammau4}
  \gamma(u) = \varphi(w) = w + t \partial w = \pi_0(\widetilde u) + t \partial \pi_0(\widetilde u) 
  = \pi_0 f (u) + t \partial f (u)
  \\=(\pi_0+t\partial)(1+\lambda\partial+g\varDelta)u
  =\pi_0(u)+\lambda\partial u+\pi_0g\varDelta u+(2\ell+1)t\partial u+t\partial g\varDelta u
  \\=\pi_0(u)+(\lambda+(2\ell+1)t)\partial u+\pi_0g\varDelta u+t\partial g\varDelta u,
\end{multline}
where we have used the identity $\partial \pi_0=\partial$ from Proposition~\ref{pi}, and the identities 
$\pi_0(\lambda\partial)=\pi_0(\lambda)\partial=\lambda\partial$, $t\partial(\lambda\partial)=t(\partial\lambda)\partial=2\ell t\partial$, which are valid because both $\pi_0$ and $\partial$ are $\SU$-linear.

Next we consider each of the four summands on the right hand side of~\eqref{gammau4} separately.

Note that $\partial_i u = u\overline{u}^i=(-1)^iu^{i+1}\mod J$. Using the formula from Proposition~\ref{pi} we obtain
\begin{equation}\label{1sum}
  \pi_0 (u) = u + \sum_{i \ge 2} \alpha_{1i}\partial_i u=
  u + \sum_{i \ge 2}(-1)^i \alpha_{1i}u^{i+1} \mod J^2.
\end{equation}
Similarly, 
\begin{equation}\label{2sum}
  (\lambda+(2\ell+1)t)\partial u=-(\lambda+(2\ell+1)t)u^2\mod J^2+tJ.
\end{equation}

By Lemma~\ref{gDelta}, we have $g\varDelta(u)=\sum_{i\ge2}\mu_i u^{i+1}$.
It follows that $g\varDelta(u)$ is an element of 
$\widetilde{\MU} \vphantom{MU}^2(\C P^\infty)$ with the property $g\varDelta(u)|_{\C P^2} = 0$. Furthermore, by Lemma~\ref{gDelta} any 
$v\in \widetilde{\MU} \vphantom{MU}^2(\C P^\infty)$ with 
$v|_{\C P^2} = 0$ has the form $g\varDelta(u)$ for some $\SU$-linear operation~$g$. Applying the projection $\pi_0$ we obtain the element 
$\pi_0g\varDelta(u)\in\widetilde W^2(\C P^\infty)$ with the same property. Therefore, 
$\pi_0g\varDelta(u)$ can be written as a power series in $w$ without quadratic part with respect to the product~$*$:
\[
  \pi_0g\varDelta(u)=\sum_{i\ge2}\omega_i*w^{*(i+1)},
\]
where $\omega_i \in W^{-2i}$ can be arbitrary. 
By~\eqref{mult}, $\omega_i*w = \omega_i w + \delta\partial \omega_i \partial w=\omega_iw\mod J^2$, where the last identity holds because $\delta$ and $\partial\omega_i$ are in $J$ for $i \ge 2$. Furthermore, 
\[
  w = \pi_0(\widetilde u) = \pi_0(u) + \lambda\partial (u) + \pi_0 g \varDelta (u) = u\mod J.
\]
It follows that
\begin{equation}\label{3sum}
  \pi_0g\varDelta(u)=\sum_{i\ge2}\omega_i u^{i+1}\mod J^2.
\end{equation}

Finally, by Lemma~\ref{gDelta}, $\partial g \varDelta = \sum \limits_{i \ge 2} \lambda_i \partial_i$, where $\lambda_i \in\MU^{2-2i} \subset J$. It follows that $\partial g \varDelta (u) = 0\mod J$ and $t \partial g \varDelta (u) = 0\mod tJ$. Substituting this together with~\eqref{1sum}, \eqref{2sum}, \eqref{3sum} in~\eqref{gammau4} we obtain the required identity $\mod J^2+tJ$.
\end{proof}

The proof of Lemma~\ref{gamma} implies the following condition on the coefficients of the series expansion of $\SU$-linear projections~$\pi$ in Proposition~\ref{su-pr}:

\begin{proposition}
Let $\pi = 1 + \sum_{i \ge 2} \lambda_i \partial_i$ be an $\SU$-linear projection $\MU\to W$. Then $\lambda_i=\alpha_{1i}+\omega_i$ modulo decomposables in $\MU^*$, where $\omega_i$ can be an arbitrary element of~$W^{-2i}$.
\end{proposition}

\begin{proof}
We have
\[
  \pi(u)=u+\sum_{i \ge 2} \lambda_i \partial_iu=u+\sum_{i \ge 2} \lambda_i u\overline u^i=
  u+\sum_{i \ge 2}(-1)^i\lambda_i u^{i+1}\mod J^2.
\]
On the other hand, using Theorem~\ref{pr1}, \eqref{stong} and~\eqref{3sum} we obtain
\[
  \pi(u)=\pi_0(u)+\pi_0f\varDelta(u)=u+\sum_{i \ge 2}\alpha_{1i}\partial_iu+
  \sum_{i \ge 2}\omega_iu^{i+1}=u+\sum_{i \ge 2}((-1)^i\alpha_{1i}+\omega_i)u^{i+1}\mod J^2.
\]
Comparing the coefficients in the above two expressions for $\pi(u)$ we obtain the result.
\end{proof}

Now we return to the formal group law $\varphi_* F_W(u, v) = u + v + \sum (\omega_{ij} + t \partial \omega_{ij}) u^i v^j$ of the theory~$\varGamma$.

\begin{lemma}\label{fgr}
In the notation of Lemma~\ref{gamma},
\begin{multline*}
\varphi_* F_W(u, v) =
  u + v - 2\bigl(\lambda + (2\ell+1)t\bigr) uv - 2\delta(2\ell+1)^2(uv^2+vu^2) + \\
  + \sum_{i\ge1,\,j\ge1} \alpha_{ij}u^i v^j + \sum_{i\ge3} \gamma_i \bigl((u+v)^i-u^i-v^i\bigr)
  \mod J^2+tJ.
\end{multline*}
\end{lemma}

\begin{proof}
We have $\varphi_* F_W (u, v) = \gamma F_U (\gamma^{-1}(u), \gamma^{-1} (v))$  by Proposition~\ref{fg}. Furthermore, 
$\gamma(u)= u - (\lambda + (2\ell+1)t) u^2 + \sum_{i \ge 3} \gamma_i u^i \mod J^2+tJ$
by Lemma~\ref{gamma}. The identity
$(F_U(x, y))^i = (x + y) ^i\mod J$ implies
\begin{equation}\label{gammaxy}
  \gamma(F_U(x, y)) = F_U(x, y) - \bigl(\lambda + (2\ell+1)t\bigr) (x+y)^2 + 
  \sum \limits_{i \ge 3} \gamma_i (x+y)^i \mod J^2+tJ.
\end{equation}

Next we need to calculate $x=\gamma^{-1}(u)$. We denote
\[
  \gamma^{-1}(u) = u+ \sum_{j \ge 2} \varepsilon_j u^j,\quad 
  \gamma(u) = u + \sum_{i \ge 2} \gamma_i u^i, \text{ where } 
  \gamma_2 = -\lambda -(2\ell+1)t.
\]
All subsequent calculations will be carried out in the ring $R=\varGamma^*/(J^2+tJ)$, that is, 
$\mod J^2+tJ$, see Construction~\ref{Gamma}. We have $\varepsilon_j\in R^{2-2j}$ and $R^{<0}R^{<-2}=0$, which implies 
$\varepsilon_2(u+\sum_{i \ge 2} \gamma_i u^i)=\varepsilon_2(u - (2\ell+1)t u^2)$ and 
$\varepsilon_j(u+\sum_{i \ge 2} \gamma_i u^i) = \varepsilon_j u$ for $j \ge 3$. Hence,
\begin{multline*}
  u = \gamma^{-1}(\gamma (u)) = u + \sum_{i \ge 2} \gamma_i u^i 
  +\sum_{j \ge 2} \varepsilon_j (u + \sum_{i \ge 2} \gamma_i u^i )^j \\=
  u + \sum_{i \ge 2} \gamma_i u^i + \varepsilon_2  (u - (2\ell+1)t u^2)^2 + 
  \sum_{j \ge 3} \varepsilon_j u^j.
\end{multline*}
Comparing the coefficients of $u^j$ we obtain
\begin{align*}
  \varepsilon_2&=-\gamma_2=\lambda+(2\ell+1)t,\\
  \varepsilon_3&=2\varepsilon_2(2\ell+1)t-\gamma_3=2(\lambda+(2\ell+1)t)(2\ell+1)t-\gamma_3=
   2(2\ell+1)^2\delta-\gamma_3,\\
  \varepsilon_j&=-\gamma_j\quad\text{for }j\ge4.
\end{align*}
Therefore,
\[
  \gamma^{-1}(u) = u +\bigl(\lambda + (2\ell+1)t\bigr)u^2 
  + \bigl(2(2\ell+1)^2\delta-\gamma_3\bigr)u^3 -\sum_{j \ge 4} \gamma_j u^j.
\]

It remains to substitute $x = \gamma^{-1}(u)$ and $y = \gamma^{-1}(v)$  in~\eqref{gammaxy}.
We have $F_U(x,y)=x+y+\sum\alpha_{ij}x^iy^j=x+y+\sum\alpha_{ij}u^iv^j$ in $R$, and similarly
$\sum_{i\ge3}\gamma_i(x+y)^i=\sum_{i\ge3}\gamma_i(u+v)^i$. For the remaining summand of~\eqref{gammaxy}, we calculate
\begin{multline*}
  \bigl(\lambda + (2\ell+1)t\bigr)(x+y)^2=\bigl(\lambda + (2\ell+1)t\bigr)
  \bigl(u+v+(\lambda + (2\ell+1)t)(u^2+v^2)\bigr)^2\\ = 
  \bigl(\lambda + (2\ell+1)t\bigr)\bigr((u+v)^2 + 2(\lambda + (2\ell+1)t) (u+v)(u^2+v^2)\bigr)
  \\= (\lambda + (2\ell+1)t) (u+v)^2 + 2(2\ell+1)^2\delta (u+v)(u^2+v^2).
\end{multline*}
Substituting these expressions in~\eqref{gammaxy} we obtain
\begin{multline*}
  \gamma (F_U(\gamma^{-1}(u),\gamma^{-1}(v))) = \gamma^{-1}(u) + \gamma^{-1}(v) + \sum   
  \alpha_{ij}u^i v^j + \sum \limits_{i \ge 3} \gamma_i (u+v)^i \\ 
  -(\lambda + (2\ell+1)t) (u+v)^2 - 2(2\ell+1)^2\delta(u+v)(u^2+v^2) 
  = u + (\lambda + (2\ell+1)t)u^2 \\
  +(2(2\ell+1)^2\delta-\gamma_3)u^3 - \sum \limits_{i \ge 4} \gamma_i u^i + 
  v + (\lambda + (2\ell+1)t)v^2 + (2(2\ell+1)^2\delta-\gamma_3)v^3
  -\sum \limits_{i \ge 4}\gamma_i v^i\\
  +\sum \alpha_{ij}u^i v^j + \sum \limits_{i \ge 3} \gamma_i (u+v)^i 
  - (\lambda + (2\ell+1)t) (u+v)^2 - 2(2\ell+1)^2 \delta (u^3 + uv^2+vu^2+v^3) \\ 
  = u + v - 2(\lambda + (2\ell+1)t) uv - 2\delta (2\ell+1)^2(uv^2+vu^2) + \sum \alpha_{ij} u^i v^j 
  + \sum \limits_{i \ge 3} \gamma_i ((u+v)^i - u^i - v^i),
\end{multline*}
as claimed.
\end{proof}

\begin{lemma}
The coefficients of the formal group law $F_W (u, v) = u+v+\sum\omega_{ij} u^i v^j$
satisfy
\begin{equation}\label{higher}
  \sum \limits_{i+j=k+1} \omega_{ij}u^i v^j = \sum \limits_{i+j=k+1} \alpha_{ij}u^i v^j + 
  \gamma_{k+1}\bigl((u+v)^{k+1} - u^{k+1} - v^{k+1}\bigr)\mod J^2
\end{equation}
for $k \ge 3$.
\end{lemma}
\begin{proof}
We have $t \partial \omega_{ij}=0\mod tJ$ for $i+j>2$, which implies
$\varphi_* F_W(u, v) = u+v+(\omega_{11}+t\partial\omega_{11})uv+\sum_{i+j>2}\omega_{ij}u^iv^j\mod J^2+tJ$. Now the required identity follows from the identity of Lemma~\ref{fgr}.
\end{proof}

For an integer $k\ge1$ let
\[
  m_k = \gcd \biggl\{{k+1 \choose i}, \, 1\le i \le k \biggr\}=
  \begin{cases}  
  1&\text{if $k+1\neq p^\ell$ for any prime $p$,}\\
  p&\text{if $k+1=p^\ell$ for some prime $p$ and integer $\ell>0$.}
  \end{cases}
\]

\begin{theorem}[{see \cite[Chapter~X]{ston68} or \cite[Theorem~6.10]{c-l-p19}}]\label{Wring}
With respect to the multiplication defined by the Stong projection $\pi_0$, the ring $W_*$ is polynomial on generators in every positive even degree
except~$4$:
\[
  W_* \cong \Z[x_1, x_k \colon k \ge 3], \quad x_1=[\C P^1], \quad x_k \in W_{2k}.
\]
The polynomial generators $x_k$ are specified by the condition $s_k(x_k) = \pm m_k m_{k-1}$ for $k\ge3$.
\end{theorem}

\begin{lemma}[\cite{buch72}]\label{coeff}
For the coefficients of the formal group law $F_W (u, v)$, we have
\[
  \gcd \bigl\{ s_{i+j-1}(\omega_{ij}) \colon i+j = k+1\bigr\} 
  = m_k \bigl ( 1+(-1)^k (k+1)+ c_k m_km_{k-1} \bigr )
\]
for $k \ge 3$, where  $c_k$ can be an arbitrary integer depending on the orientation~$w$.
\end{lemma}

\begin{proof}
Recall that $\MU_* \cong\Z[a_1, a_2, \ldots]$, $a_k \in\MU_{2k}$ and $s_k(a_k)=m_k$.

There is the following formula for the coefficients of the formal group law $F_U$ modulo decomposables:
\[
  \sum \limits_{i+j=k+1} \alpha_{ij}u^i v^j = -
  a_k \frac{(u+v)^{k+1} - u^{k+1} - v^{k+1}}{m_k}\mod J^2, 
\]
see, e.\,g.,~\cite{adam74}. Hence, $\alpha_{ij} = - \frac{{i+j\choose i}}{m_{i+j-1}} a_{i+j-1}\mod J^2$ and, in particular, $\alpha_{1j} = - \frac{j+1}{m_j} a_j\mod J^2$.
This implies that $\gamma_{k+1} = (-1)^k \alpha_{1k} + \omega_k = (-1)^{k+1} \frac{k+1}{m_k}a_k + \omega_k\mod J^2$. Substituting this in~\eqref{higher} we obtain 
\[
  \sum \limits_{i+j=k+1} \omega_{ij}u^i v^j = - \bigl(a_k + (-1)^k (k+1) a_k - m_k \omega_k\bigr) 
  \frac{(u+v)^{k+1} - u^{k+1} - v^{k+1}}{m_k}\mod J^2.
\]
It follows that
\begin{multline*}
  \gcd\bigl\{  s_{i+j-1}(\omega_{ij}) \colon i+j = k+1\bigr\}  \\
  =s_k\bigl(a_k+(-1)^k(k+1)a_k - m_k\omega_k\bigr)
  =m_k\bigl(1 +(-1)^k(k+1) - s_k(\omega_k)\bigr).
\end{multline*}
Formula \eqref{mult} implies that if an element $x \in W_{2i}$, $i \ge 3$, is decomposable in $W_*$ (with respect to an arbitrary multiplication), then its forgetful image in $\MU_*$ is also decomposable. Hence, $\omega_k = c_k x_k\mod J^2$ for an integer~$c_k$. Therefore, $s_k(\omega_k) = c_k m_km_{k-1}$, and the result follows.
\end{proof}

\begin{theorem}\label{notgen}
For any complex orientation of $W$, the coefficients of the corresponding formal group law $F_W$ do not generate the ring $W_*$.
\end{theorem}

\begin{proof}
Consider a polynomial generator $x_k$ from Theorem~\ref{Wring} for $k \ge 3$. Suppose $x_k$ lies in the ring generated by the coefficients of $F_W$. Since an element decomposable in $W_*$ is also decomposable in $\MU_*$ (in dimensions $\geq 6$), we get
\[
  s_k(x_k)=\pm \gcd \bigl\{ s_{i+j-1}(\omega_{ij}) \colon i+ j = k+1\bigr\} = \pm m_k ( 1+(-1)^k (k+1)+ c_k m_km_{k-1} ).
\]  
On the other hand, by Theorem~\ref{Wring} we have $s_k(x_k) = \pm m_km_{k-1}$. We show that there is $k\ge3$ such that the two numbers do not agree even up to a sign.

Indeed, recall that $m_k = p$ if $k+1= p^s$ for some prime $p$, and $m_k=1$ otherwise. Therefore, if $k=2^\ell$ and, in addition, $k+1 = p^s$ for an odd prime $p$, then $m_km_{k-1}=2p$, whereas $m_k ( 1+(-1)^k (k+1)+ c_k m_km_{k-1} )=2p+p(2^\ell+2c_kp)=2p(1+2^{\ell-1 } + c_kp)$. Suppose $\pm 2p=2p(1+2^{\ell-1} + c_k p)$ or, equivalently, $1+2^{\ell-1} + c_k p = \pm 1$. Since $p$ is odd, $2^{\ell-1} + c_k p \ne 0$ for any $c_k$. So, $1+2^{\ell-1} + c_k p \ne 1$. If $1+2^{\ell-1} + c_k p = - 1$, then $-2c_k p = 4 + 2^ \ell = 3 + p^s$. This is impossible for $p>3$. As a result, we obtain a contradiction in dimensions of the form $k=2^\ell=p^s-1$ for $\ell>1$.
\end{proof}

We can also prove the following result stated in~\cite{buch72}.

\begin{theorem}\label{1p}
Let $A$ be the subring of $W_*$ generated by the coefficients of the formal group law~$F_W$.  Then there is an orientation of $W$ such that $A[\frac{1}{2}]=W_*[\frac{1}{2}]$.
\end{theorem}

\begin{remark}
The proof given below would be simpler if we knew that 
$W_*$ is a polynomial ring \emph{for arbitrary $\SU$-linear multiplication} on~$W$. However, the description of Theorem~\ref{Wring} is valid only for the multiplication defined by the Stong projection.
\end{remark}

The proof is based on three lemmata. The first lemma says that 
the case $k=2^\ell=p^s-1$ considered in the proof of Theorem~\ref{notgen} is the only case when the gcd of the $s$-numbers of the coefficients of the formal group law $F_W$ does not agree 
with~$m_km_{k-1}$:

\begin{lemma}\label{cases}
If $k$ is not of the form $k=2^\ell= p^s-1$ for some odd prime $p$, then $\gcd \bigl\{ s_{i+j-1}(\omega_{ij}) \colon i+j = k+1\bigr\} = m_k m_{k-1}$ for some value of $c_k$ (see Lemma~\ref{coeff}).
\end{lemma}

\begin{proof}
By Lemma~\ref{coeff}, we need to find $c_k$ such that $1+ (-1)^k(k+1)+c_k m_k m_{k-1} = m_{k-1}$.

If $m_{k-1}=1$, then we set $c_k=(-1)^{k+1}\frac{k+1}{m_k}$, which is an integer as $m_k$ always divides $k+1$.

If $m_{k-1}=2$, then $k=2^\ell$. By assumption, $k\ne p^s-1$, so $m_k=1$. The required identity becomes $1+(2^\ell+1)+2c_k=2$, which is satisfied for $c_k=-2^{\ell-1}$.

If $m_{k-1}=p$ is an odd prime, then $k=p^s$. Hence, $m_k=1$ or~$2$. The required identity becomes $1-(p^s+1)+pc_km_k=p$, which is satisfied for 
$c_k=\frac{p^{s-1}+1}{m_k}$. The latter is an integer as $p^{s-1}+1$ is even. 
\end{proof}

\begin{lemma}\label{fermat}
If $p^s = 2^\ell+1$ for odd prime $p$ and positive integers $\ell, s$, then either $s=1$ and $\ell = 2^n$ (so $p$ is a Fermat prime), or $p=3$, $s=2$ and $\ell=3$.
\end{lemma}

\begin{proof} \emph{Case 1:} $p=3$.

There are obvious solutions $s=1, \ell=1$ and $s=2, \ell=3$.

Now suppose $s>2$, so that $\ell>3$. Then $3^s=2^\ell+1\equiv 0 \pmod 9$. 
It is easy to check that $2^\ell \equiv -1 \pmod 9$ if and only if $\ell = 6m+3$.
Then $3^s=2^\ell+1=(2^{2m+1})^3+1=(2^{2m+1}+1)(2^{4m+2}-2^{2m+1}+1)$. Hence, $2^{4m+2}-2^{2m+1}+1=3^{s'}$ and $2^{2m+1}+1 =3^{s''}$. Now $\ell>3$ implies $m >0$ and therefore $s'>1$. Hence, $2^{4m+2}-2^{2m+1}+1 \equiv 0 \pmod 9$. Similarly, $s''>1$, hence, $2^{2m+1}+1 =3^{s''} \equiv 0 \pmod 9$. The latter implies $2^{4m+2}-2^{2m+1}+1 \equiv 1 +1 +1 =3 \not \equiv 0 \pmod 9$. A contradiction.

\smallskip

\emph{Case 2:} $p>3$.


Reducing the identity $p^s=2^\ell+1$ modulo 3 we obtain $(\pm 1)^s \equiv (-1)^\ell+1 \pmod 3$. This implies that $s$ is odd and $\ell$ is even.
 
Write $p-1=a2^q$ with odd~$a$. Then 
$2^\ell +1=(a2^q+1)^s=a^s2^{qs}+\cdots+sa2^q+1$. Suppose $s>1$. Then 
$\ell>q$ and $as+2^q(a^s2^{q(s-2)}+\cdots)=2^{\ell-q}$ is even. This is a contradiction because $as$ is odd. Hence, $s=1$ and $p=2^\ell+1$.

Write $\ell = r2^n$ with odd $r$. Then $p=2^\ell+1= (2^{2^n}+1)(2^{(r-1)2^n}-\cdots+1)$. Since $p$ is prime, we obtain $p=2^{2^n}+1$ and $\ell=2^n$.
\end{proof}

\begin{lemma}\label{suff}
If $[\C P^1]\in A$ and there are elements $x_k\in A_{2k}$, $k \ge 2$, such that $s_k(x_k)=m_km_{k-1}$ up to a power of~$2$, then $A[\frac{1}{2}]=W_*[\frac{1}{2}]$.
\end{lemma}

\begin{proof}
There is a polynomial subring $\MSU_*[\frac{1}{2}] \subset W_*[\frac{1}{2}]$. This follows from the fact that any $\SU$-linear multiplication~\eqref{mult} induces the standard multiplication on~$\MSU_*$, because $\MSU_* \subset\Ker\partial$. Furthermore, $\MSU_*[\frac{1}{2}]$ coincides with $\Ker \partial = \Im \partial$ on $W_*[\frac{1}{2}]$ (see~\cite[Chapter~X]{ston68} or~\cite[Theorem~5.11]{c-l-p19}).  From~\eqref{partial} we obtain the identity $x=\frac{1}{2}(\partial([\C P^1] x) + [\C P^1]\partial x)$ for any $x\in W_*[\frac{1}{2}]$, which implies that $W_*[\frac{1}{2}]$ is a free $\MSU_*[\frac{1}{2}]$-module with basis $\{1, [\C P^1]\}$. Since $[\C P^1] \in A$ by assumption, we need to show that $\MSU_*[\frac{1}{2}] \subset A[\frac{1}{2}]$.

By the theorem of Novikov~\cite{novi62},  $\MSU_*[\frac{1}{2}]\cong\Z[\frac{1}{2}][y_2, y_3, \ldots]$, $\dim y_k=2k$. The polynomial generators $y_k$ are specified by the condition $s_k(y_k)=\pm m_k m_{k-1}$ up to a power of $2$ (see~\cite[Chapter~X]{ston68}).

Let $x_1=[\C P^1]\in A$. We can assume by induction that $\MSU_*[\frac{1}{2}]\subset A[\frac{1}{2}]$ in dimensions less than $2k$. For $x_k \in A$, let $\widetilde y_k = \frac{1}{2} \partial (x_1 x_k)=x_k - \frac{1}{2}x_1\partial x_k$. We have $\widetilde y_k \in \MSU_*[\frac{1}{2}]$. By the induction hypothesis, $\partial x_k \in A[\frac{1}{2}]$, so $\widetilde y_k \in A[\frac{1}{2}]$. Since $x_1 \partial x_k$ is decomposable, $s_k(\widetilde y_k) = s_k(x_k)=m_km_{k-1}$ up to a power of 2. It follows that $\widetilde y_k$ is a polynomial generator of $\MSU_*[\frac{1}{2}]$. Therefore, we obtain $\MSU_*[\frac{1}{2}] \subset A[\frac{1}{2}]$ by induction.
\end{proof}

\begin{proof}[Proof of Theorem \ref{1p}.]
By Lemma \ref{suff}, we need to specify an orientation of $W$ and elements $x_k\in A_{2k}$, $k \ge 2$, such that $s_k(x_k)=m_km_{k-1}$ modulo a power of~$2$.

The formula from Lemma~\ref{fgr} implies $\omega_{11}=\alpha_{11}-2\lambda= -[\C P^1] - 2\lambda$. Choosing an orientation of $W$ with $\lambda=0$ we get 
$x_1=[\C P^1]\in A$.

\smallskip

Next we need to find $x_2\in A$ with $s_2(x_2) = m_2m_1=3$ up to a power of~$2$.

The multiplication on $W_*$ is given by $a*b = ab + \delta \partial a \partial b$, $\delta = 2[V] + \omega$, $\omega \in W_4$, $[V]=[\C P^1]^2-[\C P^2]$. We have $W_4=\mathbb Z\langle9[\C P^1]^2-8[\C P^2]\rangle$, so $s_2(\omega)=24\alpha$ with $\alpha\in\Z$. Hence, $s_2(\delta) = -6+24\alpha$.

The formula from Lemma \ref{fgr} implies 
\[
  \omega_{12}=-2\delta(2\ell+1)^2+\alpha_{12}+3\gamma_3 \mod J^2.
\]  
Substituting here $2\ell=\partial\lambda=0$ and $\gamma_3=\omega_2+\alpha_{12}$ (see Lemma~\ref{gamma}), we obtain
\[   
  \omega_{12}= 4\alpha_{12}+3\omega_2-2\delta \mod J^2,
\]  
where $\omega_2 \in W_4$ can be chosen arbitrarily depending on the orientation of~$W$. Since $s_2(\alpha_{12})=3$ and $s_2(\omega_2) = 24 \beta$ with $\beta\in\Z$, we obtain $s_2(\omega_{12}) = 24-48\alpha+72 \beta$.

\emph{Case 1:} $\alpha = 3n$, $n\in\Z$.
Let $\beta = 2n$. Take $x_2=\omega_{12}\in A$. Then $s_2(x_2)=24=3\cdot 2^3$.

\emph{Case 2:} $\alpha = 3n + \varepsilon$, $n\in\Z$, $\varepsilon=1$ or~$2$. Let $\beta = -2n$. Take  $x_2=\omega_{12}+x_1*x_1\in A$. We have $x_1*x_1=(x_1)^2+4\delta$, so $s_2(x_2)=s_2(\omega_{12})+4s_2(\delta)=24(3\beta+2\alpha)
=3\cdot\varepsilon2^4$, which is equal to $3$ up to a power of~$2$.

\smallskip

It remains to choose $x_k$ for $k\ge3$. By Lemma~\ref{cases}, there is an integral linear combination $x_k$ of the coefficients $\omega_{ij}$, $i+j=k+1$, such that  $s_k(x_k)=m_k m_{k-1}$, except for the case $k+1=p^s=2^\ell+1$.

\smallskip

For the remaining case $k=2^\ell= p^s-1$, Lemma \ref{fermat} implies that either $p=3$, $s=2$ or $p=2^{2^n}+1$, $s=1$.

In the first case ($k=8$), Lemma~\ref{coeff} gives $\gcd \bigl\{  s_k(\omega_{ij}) \colon i+j = k+1\bigr\} =  3(1+9+6c_k)$. Setting $c_k=-2$, we obtain $\gcd \bigl\{  s_k(\omega_{ij}) \colon i+j = k+1\bigr\} = -6 = -m_k m_{k-1} $, as needed.

In the second case, setting $c_k = \frac{p-1}{2}-1$ we obtain $\gcd\bigl\{  s_k(\omega_{ij}) \colon i+j = k+1\bigr\}=p (p^2-2p+1)=p(p-1)^2=2^{2^{n+1}} p$,
as needed.
\end{proof}


\begin{thebibliography}{99}

\bibitem{adam74}
Adams J.\,F. \emph{Stable homotopy and generalised homology.} Chicago lectures in mathematics. The University of Chicago Press, 1974.

\bibitem{ati61}
Atiyah M.\,F. \emph{Bordism and cobordism.} Proc. Cambridge Philos. Soc.~57 (1961), 200--208.

\bibitem{baku}
Bakuradze M. \emph{Polynomial generators of $\MSU^*[1/2]$ related to classifying maps of certain formal group laws.} Preprint (2021); 
arXiv:2107.01395.

\bibitem{ba-ro20}
Barnes D.; Roitzheim C. \emph{Foundations of
stable homotopy theory.} Cambridge studies in advanced mathematics, 185. Cambridge University Press, 2020.

\bibitem{b-b-n-y00}
Botvinnik B.\,I.; Buchstaber V.\,M.; Novikov S.\,P.; Yuzvinskii S.\,A. \emph{Algebraic aspects of the theory of multiplications in complex cobordism theory}. Uspekhi Mat. Nauk~55 (2000), no.~5, 5--24 (Russian); Russian Math. Surveys~55 (2000), no.~5, 613--633 (English translation).

\bibitem{buch72}
Buchstaber V.\,M. \emph{Projectors in unitary cobordisms that are related to SU-theory}. Uspekhi Mat. Nauk~27 (1972), no.~6, 231--232 (Russian).

\bibitem{buch12}
Buchstaber [Bukhshtaber], V.\,M.
\emph{Complex cobordism and formal groups.} 
Uspekhi Mat. Nauk~67 (2012), no.~5, 111--174 (Russian); 
Russian Math. Surveys~67 (2012), no.~5, 891--950 (English translation).

\bibitem{co-fl66}
Conner P.\,E.; Floyd E.\,E. \emph{Torsion in $\SU$-bordism}. Mem. Amer. Math. Soc.~60 (1966).

\bibitem{ekmm}
Elmendorf A.\,D.; Kriz I.; Mandell M.\,A.; May J.\,P.  \emph{Rings, modules, and algebras in stable homotopy theory}. With an appendix by M.~Cole. Mathematical Surveys and Monographs,~47, American Mathematical Society, 1997.

\bibitem{land67}
Landweber P.\,S. 
\emph{Cobordism operations and Hopf algebras}. 
Trans. Amer. Math. Soc.~129 (1967), 94--110.

\bibitem{c-l-p19}
Limonchenko I.\,Yu.; Panov T.\,E.; Chernykh G.\,S. \emph{SU-bordism: structure results and geometric representatives}. Uspekhi Mat. Nauk~74 (2019), no.~3, 95--166 (Russian); Russian Math. Surveys~74 (2019), no.~3, 461--524 (English translation).

\bibitem{marg83}
Margolis H.\,R. \emph{Spectra and the Steenrod Algebra. Modules over the Steenrod algebra and the stable homotopy category.} North-Holland Mathematical Library,~29. North-Holland Publishing Co., Amsterdam, 1983.

\bibitem{novi62}
Novikov S.\,P. \emph{Homotopy properties of Thom complexes.}
Mat. Sbornik~57 (1962), no.~4, 407--442 (Russian); English translation in: \emph{Topological Library, Part 1: Cobordisms and Their
Applications}, Ser. Knots Everything,~39, World Sci. Publ., Hackensack, New Jersey, 2007, pp.~211--250.

\bibitem{novi67}
Novikov S.\,P. \emph{Methods of algebraic topology from the point of view of cobordism theory}. Izv. Akad. Nauk SSSR, Ser. Mat.~31 (1967), no.~4, 855--951 (Russian); Math. USSR-Izv.~1 (1967), no.~4, 827--913 (English translation).

\bibitem{rav84}
Ravenel D.\,C. \emph{Localization with respect to certain periodic homology theories.} Amer. J. Math.~106 (1984), no.~2, 351--414.

\bibitem{rudy98}
Rudyak Yu.\,B. \emph{On Thom spectra, orientability, and cobordism.} Springer Monographs in Mathematics. Springer-Verlag, Berlin, 1998.

\bibitem{ston68}
Stong R.\,E. \emph{Notes on cobordism theory.} Princeton University Press, Princeton, New Jersey, 1968.

\bibitem{swit75}
Switzer R.\,M. \emph{Algebraic topology--homotopy and homology}. Springer-Verlag, New York-Heidelberg, 1975.
\end{thebibliography}
\end{document}